\documentclass[11pt,reqno]{amsart}
\usepackage{}
\usepackage{bbm}
\usepackage{url}
\usepackage{stmaryrd}
\usepackage{mathrsfs}
\usepackage{cases}
\usepackage{amsfonts}
\usepackage{amssymb}
\usepackage{amsmath}
\usepackage{tikz}
\usepackage{extarrows}
\usepackage{enumerate}

\allowdisplaybreaks[4]

\newtheorem{theorem}{Theorem}[section]
\newtheorem{lemma}[theorem]{Lemma}
\newtheorem{proposition}[theorem]{Proposition}

\theoremstyle{definition}
\newtheorem{definition}[theorem]{Definition}

\newtheorem{remark}[theorem]{Remark}
\numberwithin{equation}{section}

\let\al=\alpha
\let\b=\beta
\let\g=\gamma
\let\d=\delta

\let\la=\lambda
\let\r=\rho
\let\s=\sigma

\let\om=\omega
\let\G= \Gamma

\let\Om=\Omega

\let\th=\theta

\let\ep=\epsilon
\let\va=\varphi

\def\bbR{\mathbb{R}}
\def\bbS{\mathbb{S}}
\def\bbC{\mathbb{C}}

\newcommand{\be}{\begin{equation*}}
\newcommand{\ee}{\end{equation*}}
\newcommand{\ben}{\begin{equation}}
\newcommand{\een}{\end{equation}}
\newcommand{\bn}{\begin{enumerate}}
\newcommand{\en}{\end{enumerate}}
\newcommand{\bs}{\backslash}

\def\tE{\widetilde{E}}
\def\tF{\widetilde{F}}
\def\tQ{\widetilde{Q}}

\def\cq{\mathcal Q}

\def\apq{{A_{p,\,q}}}
\def\apn{{A_p}}
\def\bmo{{{\rm BMO}(\mathbb R^n)}}

\def\cmo{{{\rm CMO}(\mathbb R^n)}}
\def\bmoa{{{\rm BMO}_{\alpha}(\mathbb R^n)}}
\def\cmoa{{{\rm CMO}_{\alpha}(\mathbb R^n)}}
\def\tcmoa{\widetilde{{\rm CMO}}_{\al}(\bbR^n)}
\def\fz{{\infty}}
\def\apn{{A_p(\mathbb R^n)}}
\def\rr{{\mathbb R}}
\def\rn{{{\rr}^n}}
\def\ls{\lesssim}
\def\lpwp{{L^p_{w^p}(\rn)}}

\def\lipa{{Lip_{\alpha}(\mathbb R^n)}}

\def\coa{\mathcal {O}_{\alpha}}
\def\tcoa{\widetilde{\mathcal {O}_{\alpha}}}
\def\co{\mathcal {O}}
\def\cfq{\mathcal {F}_{Q}}

\def\cf{\mathcal {F}}
\def\cq{\mathcal {Q}}
\def\scrm{\mathscr{M}}

\begin{document}
\title[Commutators associated with Lipschitz functions]
{Boundedness and compactness of commutators associated with Lipschitz functions}
\author{WEICHAO GUO}
\address{School of Mathematics and Information Sciences, Guangzhou University, Guangzhou, 510006, P.R.China}
\email{weichaoguomath@gmail.com}
\author{JIANXUN HE}
\address{School of Mathematics and Information Sciences, Guangzhou University, Guangzhou, 510006, P.R.China}
\email{hejianxun@gzhu.edu.cn}
\author{HUOXIONG WU}
\address{School of Mathematical Sciences, Xiamen University,
Xiamen, 361005, P.R. China} \email{huoxwu@xmu.edu.cn}
\author{DONGYONG YANG}
\address{School of Mathematical Sciences, Xiamen University,
Xiamen, 361005, P.R. China} \email{dyyang@xmu.edu.cn}

\begin{abstract}
Let $\alpha\in (0, 1]$, $\beta\in [0, n)$ and $T_{\Omega,\beta}$ be a singular or fractional integral operator
with homogeneous kernel $\Omega$. In this article, a CMO type space
${\rm CMO}_\alpha(\mathbb R^n)$ is introduced and studied. In particular,
the relationship between ${\rm CMO}_\alpha(\mathbb R^n)$ and the Lipchitz space $Lip_\alpha(\mathbb R^n)$
is discussed. Moreover, a necessary condition of restricted boundedness of the iterated commutator $(T_{\Omega,\beta})^m_b$
on weighted Lebesgue spaces via functions in $Lip_\alpha(\mathbb R^n)$,
 and an equivalent characterization of the compactness
for $(T_{\Omega,\beta})^m_b$ via functions in ${\rm CMO}_\alpha(\mathbb R^n)$ are obtained.
Some results are new even in the unweighted setting for the first order commutators.
\end{abstract}
\thanks{Supported by the NSF of China (Nos.11771358, 11471041, 11701112, 11671414, 11571289), the NSF of Fujian Province of China (Nos.2015J01025, 2017J01011)
and the China postdoctoral Science Foundation (No. 2017M612628).}

\maketitle

\section{Introduction and preliminaries}\label{s1}
Let $\b\in [0,n)$. The singular or fractional integral operator with homogeneous kernel is defined by
\begin{equation}\label{integral operator, linear case}
  T_{\Omega,\,\b}f(x):=\int_{\mathbb{R}^n}\frac{\Omega(x-y)}{|x-y|^{n-\b}}f(y)dy,
\end{equation}
where $\Omega$ is a homogeneous function of degree zero and satisfies the following mean value zero property when $\b= 0$:
\begin{equation}\label{mean value zero}
  \int_{\mathbb{S}^{n-1}}\Omega(x')d\sigma(x')=0.
\end{equation}
Let $b$ be a locally integrable function on $\bbR^n$, and let $T$ be a linear operator.
The commutator $[b,T]$ is defined by
  \be
  [b,T]f(x):=b(x)T(f)(x)-T(bf)(x)
  \ee
 for suitable functions $f$. 
 The iterated commutator $T_b^m$ with $m\geq 2$ is defined by
  \ben\label{e-itera com}
  T_b^m(f):=[b,T_b^{m-1}]f,
  \een
  where we also write $T_b^1f:=[b,T]f.$
In 1976, Coifman, Rochberg and Weiss \cite{CoifmanRochbergWeiss76AnnMath} proved that if a function $b\in {\rm BMO}(\mathbb R^n)$, then the
commutator $[b,T_{\Omega,\,0}]$  is bounded on $L^p(\mathbb R^n)$ for any $p\in(1, \infty)$; via a spherical harmonics expansion argument,
they also proved that, if $[b, R_j]$ is bounded on $L^p(\mathbb R^n)$ for every Riesz transform $R_ j$, $j=1,2,\cdots, n$, then $b\in {\rm BMO}(\mathbb R^n)$.
In 1978, using a Fourier expansion technique, Janson \cite{Janson78ArkMat} first proved that for $0<\al<1$, $b\in {Lip}_\al(\mathbb{R}^n)$ if and only if $[b,T_{\Omega,\,0}]$ with smooth kernel $\Om$ is bounded from $L^p(\mathbb{R}^n)$ to $L^q(\mathbb{R}^n)$ for $1<p<q<\infty$ with $1/q=1/p-\al/n$.
Later on, Paluszy\'{n}ski \cite{Paluszynski95IUMJ} established the corresponding result for the commutator of Riesz potential $[b,I_\beta]$.
In 2008, Hu-Gu \cite{HuGu08jmaa} provided the equivalent characterizations between $[b,T_{\Omega,\,\b}]$ with smooth kernel $\Om$
and weighted Lipschitz spaces.
Recently,
the necessity of boundedness of iterated commutators $(T_{\Omega,0})^m_b$ was proved for a rather wide class of operators,
by Lerner-Ombrosi-Rivera-R\'ios \cite{LernerOmbrosiRivera17arxiv},
by a technique in terms of the local mean oscillation. 

The study of equivalent characterization on $L^p(\rn)$-compactness of commutators $[b, T]$
of singular integral operators $T$ was initiated by Uchiyama in his remarkable work \cite{Uchiyama78TohokuMathJ}, in which he showed that the commutator $[b,T_{\Omega,0}]$ is bounded (compact resp.) on $L^p(\rn)$ if and only if
the symbol $b$ is in $\bmo$ ($\cmo$ resp.). Here and in what follows, $\cmo$ is the closure of $C_c^{\infty}(\mathbb{R}^n)$ in the ${\rm BMO}(\mathbb{R}^n)$ topology.
Moreover, Uchiyama \cite{Uchiyama78TohokuMathJ} also established the following equivalent
characterization of $\cmo$ in terms of mean value oscillation of functions, which plays a
key role in the compactness characterization of $[b,T_{\Omega,0}]$ on $L^p(\rn)$.

{\bf Theorem A}\,{\label{lemma, old characterization of CMO}
  Let $f\in \bmo$. Then $f\in \cmo$ if and only if the following three conditions hold:
  \bn
  \item $\lim\limits_{r\rightarrow 0}\sup\limits_{|Q|=r}{\mathcal{O}}(f;Q)=0$,
  \item $\lim\limits_{r\rightarrow \infty}\sup\limits_{|Q|=r}{\mathcal{O}}(f;Q)=0$,
  \item $\lim\limits_{d\rightarrow \infty}\sup\limits_{Q\cap [-d,d]^n=\emptyset}\mathcal O(f;Q)=0$.
  \en
  Here and in what follows, the symbol $Q$ means closed cubes in $\mathbb R^n$ with sides parallel to the coordinate axes,
  \be
f_Q:=\frac{1}{|Q|}\int_{Q}f(y)dy\,\,\mbox{and} \ \ \
\,\,
\mathcal O(f;Q):=\frac{1}{|Q|}\int_Q |f(x)-f_Q|dx.
\ee
}

In \cite{Wang87ChineseAnnMathSerA}, by applying the characterization of $\cmo$ in \cite{Uchiyama78TohokuMathJ},
Wang showed that the fact $b\in\cmo$ is also sufficient and necessary for the
compactness of the commutator  $[b, I_\beta]$ with fractional integral operator $I_\beta$ from $L^p(\rn)$ to
$L^q(\rn)$, where $\beta\in(0, n)$, $p,q\in(1, \fz)$ with $\frac1p=\frac1q+\frac\beta n$ and
\begin{equation*}
I_\beta f(x):=\int_\rn\frac{f(y)}{|x-y|^{n-\beta}}\,dy.
\end{equation*}
Since then, the work on compactness of commutators of singular and fractional integral operators and its applications to PDE's
have been paid more and more attention; see, for example, \cite{Iwaniec92,KrantzLi01JMAAb,ChenDingWang09PA,Taylor11,ClopCruz13AASFM,ChaffeeTorres15PA,ChenHu15CMB}
and the references therein.
We only mention that very recently, inspired by the method developed in \cite{LernerOmbrosiRivera17arxiv},
an equivalent characterization of the weighted Lebesgue space compactness for
iterated commutator $(T_{\Om,\b})^m_b$ via $\cmo$ was  obtained in
\cite{GuoWuYang17Arxiv}, where in the necessity part (see \cite[Theorem 1.4]{{GuoWuYang17Arxiv}})
the author only assume that $\Omega\in L^\infty(\mathbb S^{n-1})$ and does not change sign and is not equivalent to zero on
some open subset of $\mathbb S^{n-1}$.

Let $\al\in (0,1]$ and $\lipa$ be the Banach space of all continuous functions on $\rn$ such that $\|f\|_{\lipa}<\infty$, where
 for a continuous function $f$ on $\bbR^n$, the (homogeneous)
$\al$-order Lipschitz norm is defined by
  \be
  \|f\|_{\lipa}:=\sup_{x\neq y}\frac{|f(x)-f(y)|}{|x-y|^{\al}}.
  \ee
One of the main purposes of this article is to consider the compactness characterization of
iterated commutators $(T_{\Om,\b})^m_b$ on weighted Lebesgue spaces, where $b\in\lipa$ and $\Omega\in L^r(\mathbb S^{n-1})$
with  $r\in(1, \infty]$. To this end, we first recall the following characterization of $\lipa$ by Meyers \cite{meyers1964pams}.


\begin{definition}\label{d-bmoa}
Let $\al\in [0,1]$. The space of functions with bounded fractional mean oscillation, denoted by $\bmoa$, consists of all
$f\in L_{loc}^1(\mathbb{R}^n)$ such that
\be
\|f\|_{\bmoa}:=\sup_{Q\subset \mathbb{R}^n}\coa(f;Q)<\infty,
\ee
where
\be
\coa(f;Q):=\frac{1}{|Q|^{1+\frac{\al}{n}}}\int_Q |f(x)-f_Q|dx.
\ee
\end{definition}
In \cite{meyers1964pams}, Meyers established the following equivalent characterization of $\lipa$  in terms of
$\bmoa$.
\begin{lemma}\label{l-lip norm equiv}
  Let $\al\in (0,1]$. Then
  \be
  \lipa= \bmoa.
  \ee
  Moreover, if $f\in \lipa$, $p\in [1,\infty]$, we have
  \be
  \|f\|_{\lipa}\sim \sup_{Q}\coa(f,Q)\sim \sup_Q\frac{1}{|Q|^{\frac{\al}{n}}}\left(\frac{1}{|Q|}\int_{Q}|f(y)-f_Q|^pdy\right)^{1/p},
  \ee
  where and in what follows, the symbol $f\lesssim g$ represents that $f\leq Cg$ for some
positive constant $C$ and $f\sim g$ represents $f\lesssim g$ and $g\lesssim f$.
\end{lemma}

The following class $A_p$ was introduced by Muckenhoupt \cite{Muckenhoupt72TAMS} to
study the weighted norm inequalities of Hardy-Littlewood maximal operators, and  $\apq$ was  introduced by
Muckenhoupt--Wheeden  \cite{MuckenhouptWheeden74TAMS} to
study the weighted norm inequalities of fractional integrals, respectively.

\begin{definition}\label{d-Ap weight}
For $1<p<\infty$, the Muckenhoupt class $A_p$ is the set of locally integrable weights $\om$ such that
\be
[\om]_{A_p}^{1/p}:=\sup_{Q}\left(\frac{1}{|Q|}\int_{Q}\om(x)dx\right)^{1/p}\left(\frac{1}{|Q|}\int_{Q}\om(x)^{1-p'}dx\right)^{1/p'}<\infty,
\ee
where $\frac1p+\frac1{p'}=1$. For $1<p,q<\infty$, $1/q=1/p-\al/n$ with $0<\al<n$,
a weight function $\om$ is called an $A_{p,\,q}$ weight if
$$[\om]_{A_{p,\,q}}^{1/q}:=\sup_Q\left(\frac1{|Q|}\int_Q \om^q(x)dx\right)^{1/q}\left(\frac1{|Q|}\int_Q \om^{-p'}(x)dx\right)^{1/p'}<\fz.$$
\end{definition}

Now, we are in the position to state our first main result.

\begin{theorem}\label{theorem, necessity of boundedness}
  Let $1<p,q<\infty$, $0< \al\leq 1$, $0\leq \b<n$, $\al+\b<n$, $1/q=1/p-(m\al+\b)/n$, $m\in \mathbb{Z}^+$ and $\om\in A_{p,q}$.
  Let $\Om$ be a measurable function on $\bbS^{n-1}$, which does not change sign and is not equivalent to zero
  on some open subset of  $\bbS^{n-1}$.
  If there is $C>0$ such that for every bounded measurable set $E\subset \bbR^n$,
\begin{align}\label{e-Lpq bdd restri}
  \|(T_{\Om,\b})_b^m(\chi_E)\|_{L^q(\om^q)}\leq C(\om^p(E))^{1/p},
\end{align}
then $b\in \bmoa$.
\end{theorem}
\begin{remark}
  One can see that the method of proof of Theorem \ref{theorem, necessity of boundedness} is also valid for the case $\al=0$,
and for the weighted cases (see \cite{HuGu08jmaa,LernerOmbrosiRivera17arxiv}).
In particular, this method implies another approach of \cite[Theorem 1.1 (ii)]{LernerOmbrosiRivera17arxiv}.
\end{remark}

Denote by $\widetilde{CMO}_{\al}(\bbR^n)$ the $C_c^{\infty}(\mathbb{R}^n)$ closure in $\bmoa$.
Two natural questions arise:
In Theorem A, is it true
if we replace $\bmo$ by $\bmoa$,  $\mathcal O$ by $\coa$,
and $\cmo$ by the $\widetilde{CMO}_{\al}(\bbR^n)$? And what is the relation between
$\widetilde{CMO}_{\al}(\bbR^n)$ and the compactness of commutators?
We consider the following example for a first glimpse.

\textbf{An example.}
Let $\al\in (0,1]$.
Take
\be
\va_1(t):=sgn(t)|t|^{\al},\ \  t\in \bbR,\ \  \va(x):=\va(x_1), x\in \rn.
\ee
We first claim that $\va_1\in Lip_{\al}(\bbR)$.
Indeed, since $\al\in (0,1]$, for $t,s\in \bbR$,
\be
|t+s|^{\al}\leq |t|^{\al}+|s|^{\al},\ \ \ |t|^{\al}\leq |t+s|^{\al}+|s|^{\al},
\ee
which implies that
\be
\big||t+s|^{\al}-|t|^{\al}\big|\leq |s|^{\al}.
\ee
Thus, the function $|\cdot|^{\al}$ belongs to $Lip_{\al}(\bbR)$.
To prove that $\va_1\in Lip_{\al}(\bbR)$, we only need to verify
\be
\frac{|\va_1(t+s)-\va_1(t)|}{|s|^{\al}}\leq C.
\ee
Obviously, the above inequality is valid for $t(t+s)\geq 0$ by $|\cdot|^{\al} \in Lip_{\al}(\bbR)$.
If $t(t+s)< 0$, we have $|s|\geq |t|$. So
\be
\frac{|\va_1(t+s)-\va_1(t)|}{|s|^{\al}}\leq \frac{|\va_1(t+s)|+|\va_1(t)|}{|s|^{\al}}
=\frac{|t+s|^{\al}+|t|^{\al}}{|s|^{\al}}\leq \frac{(2^{\al}+1)|s|^{\al}}{|s|^{\al}}=2^{\al}+1.
\ee
For every $x,z\in \rn$,
\be
|\va(x+z)-\va(x)|=|\va_1(x_1+z_1)-\va_1(x_1)|\leq \|\va_1\|_{Lip_{\al}(\bbR)}|z_1|^{\al}
\leq \|\va_1\|_{Lip_{\al}(\bbR)}|z|^{\al}.
\ee
Next, we will see that the condition in (1) of Theorem A fails.
By a direct calculation, for $Q_0:=[-1/2,1/2]^n$ and any $a>0$,
we have
\be
\int_{aQ_0}\va(x)dx=a^{n-1}\int_{-a/2}^{a/2}sgn(x_1)|x_1|^{\al}dx_1=0,
\ee
and
\be
\int_{aQ_0}|\va(y)|dy=a^{n-1}\int_{-a/2}^{a/2}|x_1|^{\al}dx_1\sim a^{n+\al}.
\ee
Thus,
\be
\coa(\va,aQ_0)=\frac{1}{|aQ_0|^{1+\frac{\al}{n}}}\int_{aQ_0}|\va(y)-\va_{aQ_0}|dy
=\frac{1}{|aQ_0|^{1+\frac{\al}{n}}}\int_{aQ_0}|\va(y)|dy\sim 1.
\ee
From the above example, we have two observations:
\bn
\item For $\al\in (0,1)$, since a $\widetilde{CMO}_{\al}(\bbR^n)$ function must satisfy (1) in
Theorem A with $\coa$,
we see that $\va\in \lipa\bs \widetilde{CMO}_{\al}(\bbR^n)$. Hence,
$\tcmoa$ is a non-trivial and proper subspace of $\lipa$.
\item For $\al=1$, since the function $\va\chi_{B(0,1)}$ can be a part of certain $C_c^{\infty}(\mathbb{R}^n)$ function in $B(0,1)$,
we see that there exists a $C_c^{\infty}(\mathbb{R}^n)$ function
(belongs to $\widetilde{CMO}_{\al}(\bbR^n)$) that does not satisfy (1) in
Theorem A with $\co_1$. Hence, Lemma \ref{l-lip norm equiv} fails if we
replace $\bmo$ by $Lip_1(\bbR^n)$,  $\mathcal O$ by $\co_1$,
and $\cmo$ by the $\widetilde{CMO}_{1}(\bbR^n)$.
\en
Now, we introduce another function space, $\cmoa$, associated with $\bmoa$.
One will see that when $\al=1$, the following $\cmoa$ is the right function space for the
equivalent characterization of compact commutators.
We will use $\cmoa$ to give an answer of the second question posed above,
see Theorem \ref{theorem, characterization of compactness} below.

\begin{definition}\label{def, cmoa}
  Let $\al\in [0,1]$.
  A $\bmoa$ function $f$ belongs to $\cmoa$ if
  it satisfies the following three conditions:
  \bn
  \item $\lim\limits_{r\rightarrow 0}\sup\limits_{|Q|=r}{\mathcal{O}_{\al}}(f;Q)=0$,
  \item $\lim\limits_{r\rightarrow \infty}\sup\limits_{|Q|=r}{\mathcal{O}_{\al}}(f;Q)=0$,
  \item $\lim\limits_{d\rightarrow \infty}\sup\limits_{Q\cap [-d,d]^n}\coa(f;Q)=0$.
  \en

\end{definition}
Observe that $CMO_0(\rn)=\cmo$ by Theorem A.
In the following, we give our second main result corresponding to Theorem A.
As in the $\cmo$ case, this theorem is also a key tool for the equivalent characterization of compact commutators.
\begin{theorem}\label{theorem, characterization of cmoa}
  When $\alpha\in [0,1)$,
  we have
  \be
  \tcmoa = \cmoa.
  \ee
  When $\al=1$, $\cmoa\subsetneqq \tcmoa$.
  In fact, $\rm CMO_{1}(\rn)$ is equal to the constant space $\mathbb{C}$ containing all complex numbers with usual norm.
\end{theorem}

Based on Theorems \ref{theorem, necessity of boundedness} and \ref{theorem, characterization of cmoa}, we further have the following
result on compactness characterization of iterated commutator $(T_{\Om,\b})_b^m$.

\begin{theorem}\label{theorem, characterization of compactness}
  Let $1<p,q<\infty$, $0< \al\leq 1$, $0\leq \b<n$, $m\al+\b<n$, $1/q=1/p-(m\al+\b)/n$, $m\in \mathbb{Z}^+$.
Suppose $r'\in [1,p)$, $\om^{r'}\in A_{\frac{p}{r'},\frac{q}{r'}}$.
  Let $\Om\in L^r(\bbS^{n-1})$, which does not change sign and is not equivalent to zero
  on some open subset of  $\bbS^{n-1}$.
  The following two statements is equivalent:
  \bn
  \item $(T_{\Om,\b})_b^m$ is a compact operator from $L^p(\om^p)$ to $L^q(\om^q)$;
  \item $b\in \cmoa$.
  \en
\end{theorem}

We remark that, while we were putting the finishing touches on this manuscript,
we learned that a similar result concerning Theorem \ref{theorem, characterization of cmoa}
have been obtained independently in \cite{NogayamaSawano17MathematicalNotes},
 where the authors characterize the compactness of the commutators generated by Lipschitz functions and fractional
integral operators on Morrey spaces.
Our new contribution of the above three theorems is the following.
\bn
\item For $m=1$, Theorem \ref{theorem, necessity of boundedness} can be applied to a
much wider class of operators, compared to the known results, even for the unweighed cases.
For $m\geq 2$, Theorems \ref{theorem, necessity of boundedness} is new even for $\Om\equiv 1$ and $\om\equiv 1$.
\item
The proof of Theorem \ref{theorem, necessity of boundedness} is also valid for the case $\al=0$,
and for the weighted cases (see \cite{HuGu08jmaa,LernerOmbrosiRivera17arxiv}).
We mention that we prove Theorem \ref{theorem, necessity of boundedness} by using some idea from \cite{LernerOmbrosiRivera17arxiv}
(see also \cite{GuoWuYang17Arxiv}).
It is known that $\bmo$ and $\cmo$ can be characterized by local mean oscillation of functions
(see \cite{Stromberg79IUMJ} and \cite{GuoWuYang17Arxiv} respectively), and
 \cite[Proposition 3.1]{LernerOmbrosiRivera17arxiv} plays a crucial role
 in the study of two weighted $L^p$-boundedness of iterated commutator  via weighted $\rm BMO$ function therein,
 which is on the domination of the local mean oscillation of a function $b$
on any given cube $Q$ by the quantity $|b(x)-b(y)|$ pointwise on a subset $G$ of $Q\times P$,
where $P$ is a cube having comparable volume to $Q$;  see also \cite[Proposition 4.1]{GuoWuYang17Arxiv}.
However, since the loss of local mean oscillation in $\lipa$ with $\al\in (0,1]$,
the method in \cite{LernerOmbrosiRivera17arxiv} can not be applied to our case.
A novelty of this article lies in that we obtain a version of such domination via the so-called median value of $b$
 on two subsets of $Q\times P$, see  Proposition \ref{proposition, key lower estimates} below.
In this sense, 
our proof can be applied to the function spaces
without local mean oscillation.
\item For $\al=1$, Theorem \ref{theorem, characterization of cmoa} is new, and for $\al\in (0,1)$, our proof
of Theorem \ref{theorem, characterization of cmoa} is totally different from \cite{NogayamaSawano17MathematicalNotes}.
\item
For $m=1$, Theorem \ref{theorem, characterization of compactness} is the first result of
equivalent characterization of compact commutators with rough kernel.
For $m\geq 2$,
Theorem \ref{theorem, characterization of compactness} is new even for $\Om\equiv 1$ and $\om\equiv 1$.
Since the loss of local mean oscillation in $\lipa$, our method is quite different from \cite{GuoWuYang17Arxiv}.
\en

The rest of the paper is organized as follows. In Section \ref{s3},
we first obtain a domination result via the median value of  a given function $b$ and cube $Q$
by the quantity $|b(x)-b(y)|$ pointwise on two subsets of $Q\times P$, see  Proposition \ref{proposition, key lower estimates} below.
By making use of Proposition \ref{proposition, key lower estimates}, for any given cube $Q$ and  real-valued measurable function $b$,
we further construct two functions $f_i (i=1,2)$ related to $Q$, and
obtain a lower bound for the sum of weighted $L^q$
norm of $(T_{\Omega,\,\al})_b^m(f_i)$ over $Q$ in terms of $\coa(b,Q)$; see Proposition \ref{proposition, lower estimates} below.
Using this lower bound, we present the proof of Theorem \ref{theorem, necessity of boundedness}.

Section \ref{s2} contains the proof of Theorem \ref{theorem, characterization of cmoa}.
In Section \ref{s2.1}, we first introduce and study a kind of  Lipschitz functions $\{\mathcal F_Q\}$ associated with
a finite family $\mathcal Q$ of dyadic cubes.
In Subsection \ref{s2.2},
for any given function $f\in \cmoa$ with $\alpha\in(0, 1)$,
we further
construct a $Lip_1(\bbR^n)$ function via the features of $\lipa$ and $\{\mathcal F_Q\}$
as an approximate function of $f$ in the topology of $\lipa$.
Using this key approximate function, we complete the proof of Theorem \ref{theorem, characterization of cmoa}.

Section \ref{s5} is devoted to the proof of
the $(1)\Longrightarrow (2)$ part of Theorem \ref{theorem, characterization of compactness}, and is divided into three subsections.
In Subsection \ref{s3.2},
we establish a further lower bounded estimate closely related the necessity of compact commutators.
Here, although the local mean oscillation is lost,
the continuity of $b\in \lipa$ provides enough information for the distribution of the values of $b$.
This observation makes it possible for us to get the lower bound
for the weighted $L^q$ norm of $(T_{\Omega,\,\al})_b^m(f)$ over certain measurable set associated with $Q$
in terms of $\min\left\{\left(\coa(b,Q)\right)^{2n/\al},1\right\}\coa(b,Q)^m$,
where $f$ is a suitable function related to $Q$.
In Subsection \ref{s3.3},
for $b\in\lipa$, $\Om\in L^{\infty}(\mathbb S^{n-1})$ and any cube $Q$,
we also obtain an upper bound of the weighted $L^q$
norm of $(T_{\Omega,\,\al})_b^m(f)$ over the annulus $2^{d+1}Q\bs 2^dQ$ in terms of $2^{-\d dn/p}d^m$, where
$d\in\mathbb N$ large enough, $\delta$ is a positive constant depending on $w\in\apq$ and $f$ is aforementioned.
Using Theorem \ref{theorem, characterization of cmoa}, the upper and further lower bounds,
and a reduction of $\Om$,
we further present the proof of $(1)\Longrightarrow (2)$ part in Theorem \ref{theorem, characterization of compactness} via a contradiction argument in Subsection \ref{s3.4}.

In Section \ref{s4}, we give the proof for the $(2)\Longrightarrow (1)$ part
in Theorem \ref{theorem, characterization of compactness}.
Using a classical boundedness result of fractional integral, we reduce the proof to the case of "good kernel".

It would be helpful to clarify that in this paper, the kernel $\Om$ is assumed to be real-valued.
For $m\geq 2$, we only consider the real-valued symbol $b$. This restriction is actually implied in all the previous results of this topic.
Here, we emphasize this to avoid possible misunderstanding.

Finally, we make some conventions on notation. Throughout the paper,  for a real number $a$, $\lfloor a\rfloor$ means the biggest integer no more than $a$.
By $C$ we denote  a {positive constant} which
is independent of the main parameters, but it may vary from line to
line. Constants with subscripts do
not change in different occurrences.  For a given cube $Q$, we use $c_Q$, $l_Q$, $\mathring{Q}$ and $\chi_Q$ to denote the center, side length, interior and  characteristic function of $Q$, respectively. Moreover, we denote $Q_0:=[-1/2,1/2]^n$.
For any point $x_0\in\rn$ and sets $E, F\subset\rn$,
$E+x_0:=\{y+x_0: y\in E\}$ and $E-F:=\{x-y: x\in E, y\in F\}$.

\section{Necessity of boundedness of commutators}\label{s3}

This section is devoted to the proof of Theorem \ref{theorem, necessity of boundedness}.
Because of the loss of local mean oscillation in $\lipa$ with $\al\in (0,1]$,
the methods used in \cite{GuoWuYang17Arxiv} and \cite{LernerOmbrosiRivera17arxiv}
can not be applied here.
By using the median value of $b$, we obtain the useful lower bound associated with $\Om$ and $b$.
Then, for a real-valued measurable function $b$,
we construct two functions $f_i (i=1,2)$ related to $Q$, and
obtain a lower bound for the sum of weighted $L^q$
norm of $(T_{\Omega,\,\al})_b^m(f_i)$ over $Q$ in terms of $\coa(b,Q)$.
Using this lower bound, Theorem \ref{theorem, necessity of boundedness} is proved.


We first recall some useful properties of $A_p$ and $\apq$ weights; see \cite{MuckenhouptWheeden74TAMS,Grafakos08,HytonenPerez13AnalPDE,HolmesRahmSpencer16StudiaMath}.
Define the $A_{\infty}$ class of weights by $A_{\infty}:=\cup_{p>1}A_p$, and recall the Fujii-Wilson $A_{\infty}$ constant
\be
[\om]_{A_{\infty}}:= \sup_Q\frac{1}{\om(Q)}\int_Q M(\chi_Q\om)\,dx,
\ee
where $M$ is the Hardy-Littlewood maximal operator.
\begin{lemma}\label{l-Ap weight prop}
Let $p\in(1, \infty)$ and $w\in A_p$.
\begin{enumerate}
  \item [{\rm(i)}] For every $0<\al<1$, there exists $0<\b<1$ such that for every $Q$ and every measurable set $E\subset Q$
with $|E|\geq \al |Q|$,
\be
\om(E)\geq \b\om(Q).
\ee
  \item [{\rm (ii)}]For all $\la>1$, and all cubes $Q$,
\be
\om(\la Q)\leq \la^{np}[\om]_{A_p}\om(Q).
\ee
  \item [{\rm (iii)}]  $[\om]_{A_{\infty}}\leq c_n[\om]_{A_p}$.

  \item[{\rm (iv)}]There exists a constant $\ep_n$ only depending on $n$, such that
 if $0<\ep\leq \ep_n/[\om]_{A_\infty}$, then $\om$ satisfies the reverse H\"{o}lder  inequality that for any cube $Q$,
\be
\left(\frac{1}{|Q|}\int_{Q}\om(x)^{1+\ep}dx\right)^{\frac{1}{1+\ep}}\leq \frac{2}{|Q|}\int_{Q}\om(x)dx.
\ee

\item[{\rm(v)}] There exists a small positive constant $\ep$ depending only on $n$, $p$ and $[\om]_{A_p}$ such that
\be
\om^{1+\ep}\in A_p,\ \ \  \om\in A_{p-\ep}.
\ee

\end{enumerate}
\end{lemma}

\begin{lemma}\label{l-Apq weigh prop}
Let $1<p,q<\infty$, $1/q=1/p-\al/n$ with $0<\al<n$ and $w\in \apq$.
\begin{itemize}
  \item [{\rm(i)}] $\om^p\in\apn$, $\om^q\in A_q$ and  $\om^{-p'}\in A_{p'}$
  \item [{\rm (ii)}]
$
\om\in \apq\Longleftrightarrow \om^q\in A_{q\frac{n-\al}{n}} \Longleftrightarrow \om^q\in A_{1+\frac{q}{p'}}
\Longleftrightarrow \om^{-p'}\in A_{1+\frac{p'}{q}}.
$
\end{itemize}
\end{lemma}

In this part, one can see that median value studied by Journ\'e \cite{Journe83} plays a key role in this type of lower estimates.
Compared to \cite{LernerOmbrosiRivera17arxiv}, our method does not use the so-called local mean oscillation, and
in this sense, can be used to a much wider class of function spaces.

\begin{definition}
  By a median value of a real-valued measurable function $f$ over $Q$ we mean a possibly nonunique, real number $m_f(Q)$ such that
\be
|\{x\in Q: f(x)>m_f(Q)\}|\leq |Q|/2
\ee
and
\be
|\{x\in Q: f(x)<m_f(Q)\}|\leq |Q|/2.
\ee
\end{definition}

\begin{proposition}\label{proposition, key lower estimates}
  Let $b$ be a real-valued measurable function.
  Suppose that $\Om$ satisfies the assumption in Theorem \ref{theorem, necessity of boundedness}.
  For every $\g\in (0,1)$,
  there exist $\ep_0>0$ and $k_0>10\sqrt{n}$ depending only on $\Om$, $\g$ and $n$ such that
  the following holds.
  For every cube $Q$, there exists another cube $P$ with the same side length of $Q$ satisfying $|c_Q-c_P|=k_0l_Q$,
  and measurable sets $E_1, E_2\subset Q$ with $Q=E_1\cup E_2$, and $F_1, F_2\subset P$ with $|F_1|=|F_2|=\frac{1}{2}|Q|$, such that
  \bn
  \item $b(x)-b(y)$ do not change sign for all $(x, y)$ in $E_i\times F_i$, $i=1,2$,
  \item $|b(x)-m_b(P)|\leq |b(x)-b(y)|$ for all $(x, y)$ in $E_i\times F_i$, $i=1,2$;
  \item $\Om\left(x-y\right)$ does not change sign  for all $(x, y)$ in $Q\times P$;
  \item $|N_x\cap P|\leq \g|Q|$ for all $x\in Q$, where  $N_x:=\{y\in \bbR^n: |\Om(x-y)|<\ep_0\}$.
  \en
\end{proposition}
\begin{proof}
Without loss of generality, assume $\Omega$ is nonnegative on an open set of $\bbS^{n-1}$.
By the assumption of $\Om$, there exists a point $\th_0$ of approximate continuity of $\Omega$ such that
$\Omega(\th_0)=2\ep_0$ for some $\ep_0>0$ (see \cite[pp.46-47]{EvansGariepy92} for the definition of approximate continuity).
It follows from the definition of approximate continuity that for every
$\b\in (0,1)$, there exists a small constant $r_{\b}$ such that
\be
\s(\{\th\in B(\th_0,r_{\b})\cap \bbS^{n-1}: \Omega(\th)\geq \ep_0\})\geq (1-\b)\s(B(\th_0,r_{\b})\cap \bbS^{n-1})
\ee
Let $\G_{\b}$ be the cone containing all $x\in \bbR^n$ such that $x'\in B(\th_0,r_{\b})\cap \bbS^{n-1}$.

There exists a vector $v_{\b}:=\frac{c_n\th_0}{r_{\b}}$, such that
\be
2Q_0+v_{\b}\in \G_{\b}.
\ee
For a fixed cube $Q$, we set
$
P_{\b}:=Q-l_Q v_{\b}.
$
Thus, $|c_Q-c_{P_{\b}}|=\frac{c_n}{r_{\b}}l_Q$.
Observing that $Q-P_{\b}\subset 2l_Q Q_0+l_Q v_{\b}\subset \G_{\b}$, for $x\in Q$ we obtain that
\be
\begin{split}
  |P_{\beta}\cap N_x|
  = &
  |(P_{\beta}-x)\cap N_0|
  \\
  \leq &
  l_Q^n|(2Q_0+v_{\b})\cap (-N_0)|
    \\
  \leq &
  l_Q^n\cdot c_n\cdot |v_{\b}|^{n-1}\cdot \b r_{\b}^{n-1}
  \leq c_n\b|Q|.
\end{split}
\ee
Take $\b=\b_0$ sufficiently small such that
\be
|P_{\beta_0}\cap N_x|\leq \g|Q|,\ \    k_0:=\frac{c_n}{r_{\b_0}}>10\sqrt{n}.
\ee
Set
\be
E_1:=\{x\in Q: b(x)\geq m_b(P_{\beta_0})\},\ \ \  \ \ \ E_2:=\{x\in Q: b(x)\leq m_b(P_{\beta_0})\},
\ee
and
\be
F_1\subset \{y\in P_{\beta_0}: b(y)\leq m_b(P_{\beta_0})\},\ \ \  \ \ \ \ F_2\subset \{y\in P_{\beta_0}: b(y)\geq m_b(P_{\beta_0})\},
\ee
such that $|F_1|=|F_2|=\frac{|P_{\beta_0}|}{2}=\frac{|Q|}{2}$.
We get the desired conclusion by
\be
\begin{split}
|b(x)-b(y)|
=&
|b(x)-m_b(P_{\beta_0})|+|m_b(P_{\beta_0})-b(y)|\geq |b(x)-m_b(P_{\beta_0})|
\end{split}
\ee
for $(x,y)\in E_i\times F_i$, $i=1,2$.
\end{proof}

The proof of Theorem \ref{theorem, necessity of boundedness} is reduced to the following proposition.

\begin{proposition}\label{proposition, lower estimates}
Let $1<p,q<\infty$, $0< \al\leq 1$, $0\leq \b<n$, $m\al+\b<n$ $1/q=1/p-(m\al+\b)/n$, $m\in \mathbb{Z}^+$.
  Let $\Om$ be a measurable function on $\bbS^{n-1}$, which does not change sign and is not equivalent to zero
  on some open subset of \,$\mathbb S^{n-1}$.  Let $w\in \apq$ and $b$ be a real-valued measurable function.
 For a given cube $Q$, let $P,E_i,F_i$, $i=1,2,$ be the sets associated with $Q$ mentioned in
 Proposition \ref{proposition, key lower estimates} with $\g=\frac{1}{4}$.
 Then there exists a positive constant $C$ independent of $Q$ and $b$ such that for $f_i:=[\om^p(F_i)]^{-1/p}\chi_{F_i}$, $i=1,2,$
  \be
  \sum_{i=1,2}\|(T_{\Omega,\,\b})_b^m(f_i)\|_{L^q(\om^q,Q)}\ge C\coa(b;Q)^m.
  \ee
\end{proposition}

Now assume the conclusion of Proposition \ref{proposition, lower estimates} for the moment, we present the proof of Theorem \ref{theorem, necessity of boundedness}.
For any cube $Q$, take $P$, $f_i$ as in Proposition \ref{proposition, lower estimates}.
We then have
  \be
  \sum_{i=1,2}\|(T_{\Omega,\,\b})_b^m(f_i)\|_{L^q(\om^q,Q)}\gtrsim\coa(b;Q)^m,
  \ee
Hence, by \eqref{e-Lpq bdd restri},
\be
\coa(b;Q)^m\lesssim \sum_{i=1,2}\|(T_{\Omega,\,\b})_b^m(f_i)\|_{L^q(\om^q,Q)}
\lesssim
\sum_{i=1,2}\|f_i\|_{L^p(\om^p)}\leq 2.
\ee
This yields $b\in \bmoa$ and finishes the proof of Theorem \ref{theorem, necessity of boundedness}.

\smallskip

\begin{proof}[Proof of Proposition \ref{proposition, lower estimates}]
Using H\"{o}lder's inequality, we obtain
  \be
  \begin{split}
  \int_{Q}|(T_{\Omega,\,\b})_b^m(f_i)(x)|dx
    \leq &
  \left(\int_{Q}|(T_{\Omega,\,\b})_b^m(f_i)(x)|^q\om^q(x)dx\right)^{1/q}\left(\int_{Q}\om^{-q'}(x)dx\right)^{1/q'}.
  \end{split}
  \ee
By the H\"{o}lder inequality, the definition of $f_i$ and Proposition \ref{proposition, key lower estimates}, for $x\in E_i$,
\begin{align*}
  \left|(T_{\Omega,\,\b})_b^m(f_i)(x)\right|
  \gtrsim &
  \frac{|\om^p(F_i)|^{-1/p}}{|Q|^{1-\b/n}}|b(x)-m_b(P)|^m\int_{F_i}|\Om(x-y)|dy
  \\
  \geq &
  \frac{|\om^p(F_i)|^{-1/p}}{|Q|^{1-\b/n}}|b(x)-m_b(P)|^m\int_{F_i\bs (N_x\cap P)}|\Om(x-y)|dy
  \\
  \geq &
  \frac{\ep_0|\om^p(F_i)|^{-1/p}}{|Q|^{1-\b/n}} |b(x)-m_b(P)|^m\cdot |F_i\bs (N_x\cap P)|
  \\
  \geq &
  \frac{\ep_0|\om^p(F_i)|^{-1/p}}{|Q|^{1-\b/n}} |b(x)-m_b(P)|^m\cdot \frac{|Q|}{4}
  \\
  \sim &
  |\om^p(F_i)|^{-1/p}|Q|^{\frac{\b}{n}}|b(x)-m_b(P)|^m
  \\
  \gtrsim &
  |\om^p(P)|^{-1/p}|Q|^{\frac{\b}{n}}|b(x)-m_b(P)|^m,
\end{align*}
where we use the facts $F_i\subset P$, $|F_i|=\frac{|P|}{2}$ and $|N_x\cap P|\leq \frac{|P|}{4}$.
This and that fact $Q=E_1\cup E_2$ yield that for $x\in Q$,
\be
\left|(T_{\Omega,\,\b})_b^m(f_1)(x)\right|+\left|(T_{\Omega,\,\b})_b^m(f_2)(x)\right|
\gtrsim
|\om^p(P)|^{-1/p}|Q|^{\frac{\b}{n}}|b(x)-m_b(P)|^m\chi_Q(x).
\ee
Hence,
\begin{align*}
  &\sum_{i=1,2}\left(\int_{Q}|(T_{\Omega,\,\b})_b^m(f_i)(x)|^q\om^q(x)dx\right)^{1/q}\left(\int_{Q}\om^{-q'}(x)dx\right)^{1/q'}
  \\
 &\quad \gtrsim
  \sum_{i=1,2}\int_{Q}|(T_{\Omega,\,\b})_b^m(f_i)(x)|dx
  \\
 &\quad \gtrsim
  |\om^p(P)|^{-1/p}|Q|^{\frac{\b}{n}}\int_Q|b(x)-m_b(P)|^mdx
  \\
 &\quad \geq
  |\om^p(P)|^{-1/p}|Q|^{\frac{\b}{n}}\left(\int_Q|b(x)-m_b(P)|dx\right)^m|Q|^{1-m}
  \\
  &\quad\geq
  |\om^p(P)|^{-1/p}|Q|^{1+\frac{m\al+\b}{n}}\tcoa(b;Q)^m
  =
  |\om^p(P)|^{-1/p}|Q|^{1/p+1/q'}\tcoa(b;Q)^m.
\end{align*}
By the fact $P\subset 4k_0Q$,
we use the definition of $\apq$ and the H\"{o}lder inequality to deduce that
\begin{eqnarray*}
 && \left(\frac{1}{|Q|}\int_{P}\om^p(x)dx\right)^{1/p}\left(\frac{1}{|Q|}\int_{Q}\om^{-q'}(x)dx\right)^{1/q'}\\
 && \quad\lesssim
  \left(\frac{1}{|Q|}\int_{Q}\om^p(x)dx\right)^{1/p}\left(\frac{1}{|Q|}\int_{Q}\om^{-q'}(x)dx\right)^{1/q'}\\
 && \quad\leq
  \left(\frac{1}{|Q|}\int_{Q}\om^q(x)dx\right)^{1/q}\left(\frac{1}{|Q|}\int_{Q}\om^{-p'}(x)dx\right)^{1/p'}\lesssim 1.
\end{eqnarray*}
The above two estimates yield that
\be
\sum_{i=1,2}\|(T_{\Omega,\,\b})_b^m(f_i)\|_{L^q(\om^q)}\gtrsim \tcoa(b;Q)^m.
\ee
\end{proof}

\section{Characterization of $\cmoa$ by fractional mean oscillation}\label{s2}

In this section, we present the proof of Theorem \ref{theorem, characterization of cmoa}.
Different from the $\bmo$ case in \cite{Uchiyama78TohokuMathJ}, since a $\lipa$ function is continuous,
the simple functions can no longer be used to approximate other functions in the topology of $\lipa$.
In order to fix this situation, we find a kind of  $Lip_1(\bbR^n)$ functions which is a suitable replacement for
the simple functions; see $\{\mathcal F_Q\}$ in Proposition \ref{proposition, Lip on Q} below. This kind of $Lip_1(\bbR^n)$ functions can be constructed by some functions $\{\psi^Q\}$ defined on the vertexes of cubes.
Thanks to the $Lip_1(\bbR^n)$ functions mentioned above,
we prove Theorem \ref{theorem, characterization of cmoa}
by the geometric part of arguments in \cite{Uchiyama78TohokuMathJ} with some careful technical modifications fitting our cases.
More precisely, for any given function $f\in \cmoa$ and $\epsilon>0$, we first choose a finite family $\cq$ of dyadic cubes, define the functions $\{\psi^Q\}$ and $\{\mathcal F_Q\}$ for all $Q\in\cq$,
and further construct a function $h_\epsilon\in C_c^\infty(\rn)$ via $\{\mathcal F_Q\}$ which approximates to $f$ in the norm of $\bmoa$.

\subsection{Lipschitz function associated with cubes}\label{s2.1}
In this section, we use $V_Q$ to denote the set of all vertexes of a given cube $Q$.
By a weighted cube we mean that there exists a vertex mapping:
\be
\psi^Q: V_Q\longrightarrow \bbC.
\ee
The oscillation of a weighted cube $Q$ is defined by
\begin{align}\label{e-osci weigh cube defn}
\scrm_Q: =\inf_{c\in \bbC}\sum_{a\in V_Q}|\psi^Q(a)-c|.
\end{align}
For a point $x:=(x_1,\cdots,x_n)\in \rn$, we define the product function by
\be
x_{\Pi}: =\prod_{j=1}^nx_j.
\ee

\begin{proposition}\label{proposition, Lip on Q}
For a given weighted cube $Q\subset \mathbb R^n$ with vertex mapping $\psi^Q$, we have following properties:
  \bn
  \item There is a unique function $\cfq$ defined on $\bbR^n$ satisfying:
    \bn
    \item $\cfq=\psi^Q$ on $V_Q$,
    \item $\cfq$ is linear for each component $x_j$, $j=1,2,\cdots,n$, when the other $n-1$ components are fixed.
    \en
    Moreover, this unique function associated with weighted cube $Q$ can be expressed by
\begin{align}\label{e-F_Q represen}
    \cfq(x)=\sum_{a\in V_Q}\frac{(2c_Q-x-a)_{\Pi}}{(2c_Q-2a)_{\Pi}}\psi^Q(a)\chi_Q(x).
    \end{align}
  \item Let $\tQ$ be a cube on $\bbR^n$, or on the m-dimensional hyperplane $P_m$ with $m<n$.
  Set $\psi^{\tQ}:=\cfq|_{V_{\tQ}}$. Then
  \be
  \cf_{\tQ}=\cfq|_{P_m}.
  \ee
  \item $\cfq$ is smooth with the following
  Lipschitz bound:
\be
\left|\frac{\partial\cfq(x)}{\partial x_i}\right|\leq |Q|^{-1/n}\scrm_Q,\ \ x\in Q, \ \  i=1,2,\cdots, n.
\ee
  \en

\end{proposition}
\begin{proof}
We first prove property (1).
By a standard translation and dilation argument, we only need to deal with $Q:=[0,1]^n$.
Since $\cfq$ is linear for each component mentioned above, we decompose $\cfq$ by
\begin{align*}
\cfq(x_1,x_2,\cdots,x_n)
=&
(1-x_1)\cfq(0,x_2,\cdots,x_n)+x_1\cfq(1,x_2,\cdots,x_n)
\\
=&
\sum_{\ep_1=0,1}\frac{1-x_1-\ep_1}{1-2\ep_1}\cfq(\ep_1,x_2,\cdots,x_n)
\\
=&
\sum_{\ep_1=0,1}\frac{1-x_1-\ep_1}{1-2\ep_1}\sum_{\ep_2=0,1} \frac{1-x_2-\ep_2}{1-2\ep_2}\cdot\cfq(\ep_1,\ep_2,x_3,\cdots,x_n)
\\
=&
\sum_{(\ep_1,\ep_2)\in\{0,1\}^2}\frac{1-x_1-\ep_1}{1-2\ep_1}\cdot \frac{1-x_2-\ep_2}{1-2\ep_2}\cdot\cfq(\ep_1,\ep_2,x_3,\cdots,x_n)
\\
=&\cdots =
\sum_{\ep\in\{0,1\}^n}\frac{(1-x-\ep)_{\Pi}}{(1-2\ep)_{\Pi}}\cfq(\ep)
=\sum_{\ep\in\{0,1\}^n}\frac{(1-x-\ep)_{\Pi}}{(1-2\ep)_{\Pi}}\psi^Q(\ep),
\end{align*}
where we use property $(a)$ in the last equality.

Next, we verify property (2).
Note that both $\cf_{\tQ}$ and $\cfq$ are linear for each component on $P_m$.
Moreover, they share the same value at each vertex in $V_{\tQ}$. By property (1), they must be equal.

Finally, we proceed to the proof of (3).
By a standard translation and dilation argument, we only need to deal with $Q:=[0,1]^n$.
In this case, write
\be
\begin{split}
  \cfq(x)
  =
  \sum_{\ep\in\{0,1\}^n}\frac{(1-x-\ep)_{\Pi}}{(1-2\ep)_{\Pi}}\psi^Q(\ep)
  =
  \sum_{\ep\in \{0,1\}^n}\prod_{j=1}^n\frac{1-x_j-\ep_j}{1-2\ep_j}\psi^Q(\ep).
\end{split}
\ee
Note that
\begin{align}\label{e-ident formu prod}
  \sum_{\ep\in \{0,1\}^n}\prod_{j=1}^n\frac{1-x_j-\ep_j}{1-2\ep_j}
  =
  \prod_{j=1}^n\left(\frac{1-x_j-0}{1-0}+\frac{1-x_j-1}{1-2}\right)=1.
\end{align}
Hence, for any $x\in \bbR^n $ and $c\in\mathbb C$,
\be
\begin{split}
  \frac{\partial\cfq(x)}{\partial x_i}
  = &
  \frac{\partial(\cfq(x)-c)}{\partial x_i}
  =
  \frac{\partial(\sum_{\ep\in \{0,1\}^n}\prod_{j=1}^n\frac{1-x_j-\ep_j}{1-2\ep_j}(\psi^Q(\ep)-c))}{\partial x_i},
\end{split}
\ee
which yields that
\be
\left|\frac{\partial\cfq(x)}{\partial x_i}\right|
\leq
\sum_{\ep\in \{0,1\}^n}|\psi^Q(\ep)-c|,\ \ \ x\in Q.
\ee
The desired conclusion follows by taking infimum over $c$.
\end{proof}

The following is a useful regularity proposition in the proof of Theorem \ref{theorem, characterization of cmoa}.

\begin{proposition}\label{proposition, regularity}
  Let $\al\in (0,1]$, $Q, \tQ$ be two cubes satisfying that $\tQ\subset Q$.
  Suppose that $\coa(f;P)<\epsilon$
  for all cubes $P\subset Q$ with $|P|\geq |\tQ|$.
  Then
  \be
  |f_Q-f_{\tQ}|\leq C|Q|^{\frac{\al}{n}}\ep,
  \ee
  where the constant $C$ is independent of $Q$, $\tQ$ and $f$.
\end{proposition}

\begin{proof}
  Take $N:=\lfloor\log_2\frac{|Q|}{|\tQ|}\rfloor$. We can find a sequence of cubes $\{Q_i\}_{i=1}^N$ satisfying
\be
Q\supset Q_1\supset Q_2\supset \cdots \supset Q_N\supset \tQ,\ \  |Q_i|=(1/2)^i|Q|\ \ (i=1,2,\cdots,N)
\ee
and
\be
(1/2)^N|Q|\geq |\tQ|> (1/2)^{N+1}|Q|.
\ee
By the assumption of Proposition \ref{proposition, regularity} and the choice of $Q_1$,
\be
\begin{split}
  |f_Q-f_{Q_1}|\leq & \frac{1}{|Q_1|}\int_{Q_1}|f(y)-f_Q|dy
  \leq
  \frac{|Q|^{1+\frac{\al}{n}}}{|Q_1|}\frac{1}{|Q|^{1+\frac{\al}{n}}}\int_{Q}|f(y)-f_Q|dy
  \leq
  2|Q|^{\frac{\al}{n}}\ep.
\end{split}
\ee
A similar argument yields that
\be
|f_{Q_j}-f_{Q_{j+1}}|\leq 2|Q_j|^{\frac{\al}{n}}\ep\ \ \text{for}\ j=1,2,\cdots, N-1,
\ee
and
\be
|f_{Q_N}-f_{\tQ}|\leq 2|Q_N|^{\frac{\al}{n}}.
\ee
Hence,
\be
\begin{split}
  |f_Q-f_{\tQ}|
  \leq &
  |f_Q-f_{Q_1}|+\sum_{j=1}^{N-1}|f_{Q_j}-f_{Q_{j+1}}|+|f_{Q_N}-f_{\tQ}|
  \\
  \leq &
  \left(2\sum_{j=1}^N|Q_j|^{\frac{\al}{n}}+2|Q|^{\frac{\al}{n}}\right)\ep
  \leq C|Q|^{\frac{\al}{n}}\ep.
\end{split}
\ee

\end{proof}

\subsection{Proof of Theorem \ref{theorem, characterization of cmoa}}\label{s2.2}
The case of $\alpha=0$ is due to Uchiyama \cite{Uchiyama78TohokuMathJ}, see Theorem A
in Section \ref{s1}.

Now, we start our proof for $\al\in (0,1)$.
Obviously, the conditions (1), (2) and (3) in Definition \ref{def, cmoa} are valid if $f\in C_c^{\infty}(\mathbb R^n)$.
For a fixed function $f\in \tcmoa$ and every $\ep>0$, we take a function $g\in C_c^{\infty}(\mathbb R^n)$ such that
\be
\|f-g\|_{\bmoa}<\ep.
\ee
Then,
\be
\begin{split}
  \limsup\limits_Q \coa(f;Q)
  \leq &
  \limsup\limits_Q \coa(f-g;Q)+\limsup\limits_Q \coa(g;Q)
  \\
  \leq &
  \|f-g\|_{\bmoa}+\limsup\limits_Q \coa(g;Q)
  \leq \ep,
\end{split}
\ee
where the cubes $Q$ vary in the way of that mentioned in (1)-(3).
This completes the proof of $\tcmoa \subset \cmoa$.

For the proof of $\cmoa\subset \tcmoa$, roughly speaking,
the geometric part of the argument in \cite{Uchiyama78TohokuMathJ} still works in our case.
However, the analytical part of argument in \cite{Uchiyama78TohokuMathJ} does not work again.
More precisely, since a $\lipa$ function is continuous,
it is clear that a simple function such as in \cite{Uchiyama78TohokuMathJ} can not be used
as an approximation function in our case.

Now, we start our proof for $\cmoa\subset \tcmoa$.
Assume that $f$ satisfies conditions (1)-(3) of Theorem \ref{theorem, characterization of cmoa}.
We will show that
for a fixed small number $\epsilon>0$, there exists a function $g_\epsilon\in \lipa\cap Lip_1(\bbR^n)$ and $h_\epsilon\in  C_c^{\infty}(\bbR^n)$, such that
\ben\label{proof, 6}
\|f-g_{\ep}\|_{\lipa}
\leq C\ep
\een
and
\begin{align}\label{e-h-e}
\|g_{\ep}-h_\epsilon\|_{\lipa}
\leq C\ep,
\end{align}
where $C$ is independent of $f$, $g_\epsilon$ and $h_\epsilon$.

We prove \eqref{proof, 6} by the following two steps.

{\bf Step I} To define the function $g_\epsilon$, we first introduce a finite family $\cq$ of closed dyadic cubes and define the vertex mappings $\{\psi^Q\}$
for all $Q\in \cq$ and auxiliary functions $\{g_{\epsilon,m}\}$ as follows.
By conditions (1)-(3), there exist three integers
depending on $\ep$, denoted by $i_{\ep}$, $j_{\ep}$ and $k_{\ep}$ respectively, satisfying $i_{\ep}+3\leq k_{\ep}$,
\be
\sup\{\coa(f;Q): |Q|\leq 2^{(i_{\ep}+2)n}\}<\ep,\ \sup\{\coa(f;Q): |Q|\geq 2^{j_{\ep}n}\}<\ep,
\ee
and
\be
\sup\{\coa(f;Q): Q\cap R_{k_{\ep}}=\emptyset\}<\ep,
\ee
where for $i\in\mathbb Z$, $R_i:=[-2^{i},2^i]^n$.
Moreover, let
\be
d_1:=d_1(\epsilon):=k_{\ep}+1
\ee
and for any integer $m\ge d_1$,
$$P_m:=R_{i_\epsilon+m-d_1-1}=\left[-2^{i_{\ep}+m-d_1-1},2^{i_{\ep}+m-d_1-1}\right]^n.$$
Using condition (2), we can find a sufficient large integer $d_2\geq j_{\ep}$
such that for all cubes $Q$ such that $|Q|\geq |P_m|=2^{(i_\epsilon+m-d_1)n}$ 
with $m\geq d_2$,
\ben\label{proof, 1}
\coa(f;Q)\leq \left(\frac{|P_{m}|}{|R_{m}|}\right)^{\frac{1}{n}}\ep.
\een
Now we consider a finite family $\cq$ of closed dyadic cubes in $R_{d_2+1}$ as follows.
For $x\in R_{d_1}$, $Q_x$ means the closed dyadic cubes of side length $2^{i_{\ep}}$ that contain $x$.
If $x\in R_m\bs R_{m-1}$ for $m>d_1$, $Q_x$ means the closed dyadic cubes with side length $2^{i_{\ep}+m-d_1}$ that contain $x$.
Observe that for any $x\in\rn$, the number of cubes $Q_x$ is not more than $2^n$.
Take
$$\cq:=\{Q_x:\,\, Q_x\subset R_{{d_2+1}}\}, \quad\, \cq_{d_1}:=\{Q\in \cq:\,\, Q\subset R_{d_1}\}$$
and
$$\cq_m:=\{Q\in \cq:\,\, Q\subset R_{m}\bs \mathring{R_{m-1}}\},\,d_1+1\leq m\leq {d_2+1}.$$
Then by the definitions of $i_\epsilon$ and $k_\epsilon,$ we see that for any cube $\tQ\subset 4Q$ with
$Q\in \cq$,
\begin{align}\label{e-osci dya cub}
\mathcal O(f; \tQ)<\epsilon,
\end{align}
which via Proposition \ref{proposition, regularity} further implies that for any $Q\in \cq_m$ and $a\in (V_Q\setminus \partial R_m)$,
\begin{align}\label{e-mean value regu Qm}
\left|f_{4Q}-f_{Q}\right|,
\left|f_{4Q}-f_{P_m+a}\right|\ls\mathcal O(f; 4Q)|Q|^{\frac\alpha n}\ls\epsilon |Q|^{\frac\alpha n}.
\end{align}

Next, we add weights on the vertexes of $Q_x\in \cq$. For $Q\in \cq_{{d_2+1}}$ and $a\in V_Q$, set
\begin{align}\label{e-psi defn outer}
\psi^Q(a):=A_{d_2}:=\frac{\sum_{b\in \bigcup\limits_{Q\in \cq_{d_2}}V_Q}f_{P_{d_2}+b}}{|\bigcup\limits_{Q\in \cq_{d_2}}V_Q|},
\end{align}
%
where $|\bigcup\limits_{Q\in \cq_{d_2}}V_Q|$ is the cardinality of $\bigcup\limits_{Q\in \cq_{d_2}}V_Q$.

For $y\in\rn$, define
\be
g_{\ep,{d_2+1}}(y):=
\begin{cases}
  \cfq(y), &\exists Q\in \cq_{{d_2+1}} \,\mbox{such\,that}\, Q\ni y,
  \\
  0, &\mbox{others,}
\end{cases}
\ee
where $\mathcal F_Q$ is the function associated with $\psi^Q$ as in Proposition \ref{proposition, Lip on Q}.
For any $y\in R_{{d_2+1}}\setminus \mathring{R_{{d_2}}}$, if there exist $Q,\, \widetilde Q\in \cq_{{d_2+1}}$ such that $y\in Q\cap \widetilde Q$, then \eqref{e-psi defn outer} and property (2) in Proposition \ref{proposition, Lip on Q} imply that
$\mathcal F_{Q}(y)=\mathcal F_{\widetilde Q}(y)$. Thus, the function $g_{\ep,{d_2+1}}$ is well-defined.

Next, for $m:=d_2, d_2-1, d_2-2,\cdots,d_1$ and $Q\in \cq_{m}$, we set
\begin{align}\label{e-psi defn inn}
\psi^Q(a):=
\begin{cases}
  g_{\ep,m+1}(a), &a\in V_Q\cap \partial R_{m},
  \\
  f_{P_m+a}, &a\in V_Q\bs \partial R_{m},
\end{cases}
\end{align}
and
\begin{align}\label{e-g-m defn inn}
g_{\ep,m}(y):=
\begin{cases}
  \cfq(y), &\exists Q\in \cq_{m} \,\mbox{such\,that}\, Q\ni y,
  \\
  0, &\mbox{others.}
\end{cases}
\end{align}
By the definition of $g_{\ep,m}$, we further set
\begin{align}\label{e-g defn}
g_{\ep}(y):=
\begin{cases}
  g_{\ep,m}(y), &\exists Q\in \cq_{m}, d_1\leq m\leq {d_2+1}, \,\mbox{such\,that}\, Q\ni y,
  \\
  A_{d_2}, &y\notin R_{{d_2+1}}.
\end{cases}
\end{align}
Thanks to the property (2) in Proposition \ref{proposition, Lip on Q},
the functions $g_{\ep,m} (d_1\leq m\leq d_2+1)$ and $g_{\ep}$ are well-defined.
Moreover, to show $g_\epsilon\in\lipa$, we first claim that for all $Q\in \cq$,
\ben\label{proof, 3}
\scrm_Q\ls|Q|^{\frac{\al}{n}}\ep,
\een
where $\scrm_Q$ is as in \eqref{e-osci weigh cube defn}.

In fact, for any $Q\in \cq_{d_2+1}$, by \eqref{e-osci weigh cube defn} and \eqref{e-psi defn outer}, it is trivial that
\begin{align}\label{e-osci weigh est o}
\scrm_Q\le \sum_{a\in V_{Q}}\left|\psi^{Q}(a)-A_{d_2}\right|=0.
\end{align}
On the other hand, for any $Q\in \cq_{d_2}$ and $a, b\in V_Q$,
by \eqref{proof, 1}, \eqref{e-osci dya cub} and Proposition \ref{proposition, regularity}, we see that
\begin{align}\label{e-regu mean value}
\left|f_{P_{d_2}+b}-f_{P_{d_2}+a}\right|&\ls \mathcal O_\alpha(f; 4Q))|Q|^{\frac\alpha n}\ls  \left(\frac{|Q|}{|R_{d_2}|}\right)^{\frac{1}{n}}\ep |Q|^{\frac\alpha n}.
\end{align}
Now for any $Q\in\cq_{d_2}$,
 from \eqref{e-regu mean value}
, it follows that
for all $a\in V_Q$,
\begin{align}\label{e-mean value regu}
\left|f_{P_{d_2}+a}-A_{d_2}\right|
\ls &
\left(\frac{|Q|}{|R_{d_2}|}\right)^{\frac{1}{n}}\ep |Q|^{\frac\alpha n}
\left(\frac{|R_{d_2}|}{|Q|}\right)^{\frac{1}{n}}
\sim
|Q|^{\frac\alpha n}\ep.
\end{align}
This via \eqref{e-psi defn inn} and \eqref{e-g-m defn inn} implies that
\begin{align}\label{e-osci weigh est-m}
\scrm_Q&\le \sum_{a\in (V_{Q}\cap \partial R_{d_2})}\left|\psi^{Q}(a)-A_{d_2}\right|+\sum_{a\in (V_{Q}\setminus \partial R_{d_2})}\left|\psi^{Q}(a)-A_{d_2}\right|\nonumber\\
&=\sum_{a\in (V_{Q}\setminus \partial R_{d_2})}\left|f_{P_{d_2}+a}-A_{d_2}\right|\ls |Q|^{\frac{\al}{n}}\ep.
\end{align}

Finally, for any $Q\in \bigcup_{m=d_1}^{d_2-1}\cq_m$,
from \eqref{e-psi defn inn}, \eqref{e-g-m defn inn}, \eqref{e-osci dya cub}, \eqref{e-mean value regu Qm}, Proposition \ref{proposition, regularity}
and the fact $4Q\subset 4\tQ$,
we deduce that
\begin{align}\label{e-osci weigh est-i}
\scrm_Q&\leq
\sum_{a\in (V_Q\cap \partial R_m)}\left|\psi^Q(a)-f_{4Q}\right|
+\sum_{a\in (V_Q\setminus\partial R_m)}\left|\psi^Q(a)-f_{4Q}\right|\nonumber\\
&\leq\sum_{a\in (V_Q\cap \partial R_m)}\left(\left|\mathcal F_{\widetilde Q}(a)-f_{4\tQ}\right|+\left|f_{4\tQ}-f_{4Q}\right|\right)
+\sum_{a\in (V_Q\setminus \partial R_m)}\left|f_{P_{m}+a}-f_{4Q}\right|
\nonumber\\
&\ls\sum_{a\in (V_Q\cap \partial R_m)}\left(\sum_{b\in V_{\widetilde Q}}\left|f_{P_{m+1}+b}-f_{4\tQ}\right|+\left|f_{4\tQ}-f_{4Q}\right|\right)
+\sum_{a\in (V_Q\setminus \partial R_m)}\left|f_{P_{m}+a}-f_{4Q}\right|\nonumber\\
&\ls \sum_{a\in (V_Q\cap \partial R_m)}\ep|4\tQ|^{\frac\alpha n}
+\sum_{a\in (V_Q\setminus \partial R_m)}\ep|4Q|^{\frac\alpha n}
\ls |Q|^{\frac\alpha n}\epsilon,
\end{align}
where in the equality, $\widetilde Q\in \cq_{m+1}$ such that $\widetilde Q\ni a$.
Hence, the claim \eqref{proof, 3} follows from \eqref{e-osci weigh est o}, \eqref{e-osci weigh est-m} and \eqref{e-osci weigh est-i}.
 Furthermore, by the claim \eqref{proof, 3} and Proposition \ref{proposition, Lip on Q} (3), for $j=1,2,\cdots, n$, $Q\in \cq$ and $y\in Q,$
\ben\label{proof, 2}
\left|\frac{\partial\cf_{Q}(y)}{\partial y_j}\right|\leq |Q|^{-1/n}\scrm_{Q}\ls|Q|^{\frac{\al-1}{n}}\ep.
\een

Now we show that $g_{\ep}\in Lip_1(\bbR^n)$.
For any two points $x_0,y_0\in \bbR^n$, one can choose finite points $x_1, x_2,\cdots, x_N$ where $N$ is independent of $x_0,y_0$, such that
$x_N:=y_0$, $|x_0-y_0|=\sum_{j=0}^{N-1}|x_j-x_{j+1}|$,
 and for any fixed $j=0,1,\cdots, N-1$, each pair of points $x_j,x_{j+1}$ both belong to a cube $Q\in \cq$ or $\mathring{R_{{d_2+1}}}$.
 Since by \eqref{proof, 2}, $g_{\ep}$ is a Lipschitz function on every cubes $Q\in \cq$, we obtain
\begin{align}\label{e-g lip pro}
|g_{\ep}(x_0)-g_{\ep}(y_0)|
\leq &
\sum_{j=0}^{N-1}\left|g_{\ep}(x_j)-g_{\ep}(x_{j+1})\right|
\leq
\sup_{Q\in \cq}\sup_{x\in Q}|\nabla\cfq(x)| \sum_{j=0}^{N-1}|x_j-x_{j+1}|\nonumber\\
=&\sup_{Q\in \cq}\sup_{x\in Q}|\nabla\cfq(x)|\cdot |x_0-y_0|.
\end{align}
We see that $g_{\ep}$ is a bounded function in $Lip_1(\bbR^n)$. This implies that $g_{\ep}\in \lipa$ for all $\al\in (0,1)$.


{\bf Step II} We now show \eqref{proof, 6}. For any cube $Q$, define
\be
\tcoa(f;Q):=\inf_{c\in \mathbb{C}}\frac{1}{|Q|^{1+\frac{\al}{n}}}\int_Q |f(x)-c|dx.
\ee
It is easy to see that
\be
\tcoa(f;Q) \leq \coa(f; Q) \leq 2\tcoa(f; Q).
\ee
Then to show \eqref{proof, 6}, it suffices to show that
\begin{align}\label{e-osci upp bdd}
\tcoa(f-g_{\ep};Q)\ls\ep
\end{align}
for all cubes $Q$ on $\bbR^n$.
This part is divided into following three cases.

\textbf{Case 1.} $Q\subset R_{{d_2+1}}$. Let $\mathcal D_Q:=\{Q_x: Q_x\cap Q\neq \emptyset\}.$ We further consider the following two cases.

Subcase (i)
$ \max\{l_{Q_x}: Q_x\in \mathcal D_Q\}\geq 2l_Q.$
In this case, the number of cubes $Q_x\in\mathcal D_Q$ is not more than $2^n$ and $l_{Q_x}\ge l_Q$ for any $Q_x\in \mathcal D_Q$.
If $Q\cap R_{d_1}\neq \emptyset$, then $|Q|\leq 2^{ni_\ep}$.
From this and the definition of $d_1$, we obtain $\coa(f;Q)<\ep$.
By \eqref{proof, 2}, we deduce that
\begin{align*}
\frac{1}{|Q|^{1+\frac{\al}{n}}}\int_Q |g_{\ep}(c_Q)-g_{\ep}(y)|dy
\ls &\frac{1 }{|Q|^{1+\frac{\al}{n}}}
\sum_{Q_x\in \mathcal D_Q}|Q_x|^{\frac{\al-1}{n}}\ep\cdot\int_{Q\cap Q_x}|y-c_Q|dy
\\
\ls &\frac{1}{|Q|^{1+\frac{\al}{n}}}
\sum_{Q_x\in \mathcal D_Q}|Q_x|^{\frac{\al-1}{n}}\ep\cdot  |Q|^{1+\frac{1}{n}}
\\
\sim &\ep
\sum_{Q_x\in \mathcal D_Q}\left(\frac{|Q|}{|Q_x|}\right)^{\frac{1-\al}{n}}\ls\ep.
\end{align*}

Thus,
\be
\tcoa(f-g_{\ep};Q)\leq \coa(f;Q)+\frac{1}{|Q|^{1+\frac{\al}{n}}}\int_Q |g_{\ep}(c_Q)-g_{\ep}(y)|dy\ls\ep.
\ee
\\
If $Q\cap R_{d_1}= \emptyset$, by the definition of $d_1$ we have $\coa(f;Q)<\ep$.
By a similar argument, we also have
\be
\tcoa(f-g_{\ep};Q)\ls\ep.
\ee

Subcase (ii)  $\max\{l_{Q_x}: Q_x\in \mathcal D_Q\}< 2 l_Q.$ In this subcase, we first show that
for any $Q_x\in\mathcal D_Q$ and $y\in Q_x$,
\begin{align}\label{e-mean val minu FQ}
|f_{Q_x}-\cf_{Q_x}(y)|\ls|Q_x|^{\frac{\al}{n}}\ep.
\end{align}
In fact, for any $Q_x\subset R_{d_2}$ and $y\in Q_x$, we take $a\in (v_{Q_x}\bs \partial R_{d_2})$ and write
\be
\begin{split}
  |f_{Q_x}-\cf_{Q_x}(y)|
  \leq &
  |f_{Q_x}-f_{4Q_x}|+|f_{4Q_x}-\cf_{Q_x}(a)|+|\cf_{Q_x}(a)-\cf_{Q_x}(y)|
  \\
  = &
  |f_{Q_x}-f_{4Q_x}|+|f_{4Q_x}-f_{P_{d_2}+a}|+|\cf_{Q_x}(a)-\cf_{Q_x}(y)|
  \\
  \ls &
  |Q_x|^{\frac{\al}{n}}\ep+|a-y|\cdot |Q_x|^{\frac{\al-1}{n}}\ep
  \ls
  |Q_x|^{\frac{\al}{n}}\ep,
\end{split}
\ee
where we use \eqref{e-mean value regu Qm}, \eqref{proof, 2} and the fact
$\cf_{Q_x}(a)=\psi^{Q_x}(a)=f_{P_{d_2}+a}$.
%


%
For any $Q_x\in \cq_{d_2+1}$ and $y\in Q_x$,  arguing as \eqref{e-mean value regu}, we also have
\begin{align}\label{proof, 4}
|f_{Q_x}-\cf_{Q_x}(y)|
=|f_{Q_x}-A_{d_2}|
\ls|Q_x|^{\frac{\al}{n}}\ep.
\end{align}
Thus, \eqref{e-mean val minu FQ} holds.

By \eqref{e-g defn}, \eqref{e-g-m defn inn}, \eqref{e-osci dya cub} and \eqref{e-mean val minu FQ}, we further write
\begin{align*}
  \int_Q|f(y)-g_{\ep}(y)|dy
  \leq &
\sum_{Q_x\in \mathcal D_Q}\left(\int_{Q_x}|f(y)-f_{Q_x}|dy+\int_{Q_x}|f_{Q_x}-\cf_{Q_x}(y)|dy\right)
 \\
  \leq &
  \sum_{Q_x\in \mathcal D_Q}\left(|Q_x|^{1+\frac{\al}{n}}\ep+\int_{Q_x}|f_{Q_x}-\cf_{Q_x}(y)|dy\right)\\
%
  \ls &
  \sum_{Q_x\in \mathcal D_Q}|Q_x|^{1+\frac{\al}{n}}\ep
  \ls\left(\sum_{Q_x\in \mathcal D_Q}|Q_x|\right)^{1+\frac{\al}{n}}\ep
  \ls|Q|^{1+\frac{\al}{n}}\ep.
\end{align*}
This implies that
\be
\tcoa(f-g_{\ep};Q)\leq \frac{1}{|Q|^{1+\frac{\al}{n}}}\int_Q|f(y)-g_{\ep}(y)|dy\ls\ep.
\ee

\textbf{Case 2.} $Q\subset R_{d_2}^c$.
Note that $g_{\ep}(y)=A_{d_2}$ for any $y\in R_{d_2}^c$. Then
\be
\tcoa(f-g_{\ep};Q)\leq \frac{1}{|Q|^{1+\frac{\al}{n}}}\int_Q|f(y)-f_{Q}|dy\ls\ep.
\ee

\textbf{Case 3.} $Q\cap R_{d_2}\neq \emptyset$, $Q\cap R_{{d_2+1}}^c\neq \emptyset$. Note that $l_Q\ge \frac12l_{R_{d_2}}$.
By this, \eqref{proof, 1} and the definition of $g_\epsilon$, we write
\begin{align*}
&\int_{Q}|f(y)-g_{\ep}(y)-f_Q+A_{d_2}|dy
\\
&\quad\leq
\int_{Q\bs R_{d_2}}|f(y)-g_{\ep}(y)-f_Q+A_{d_2}|dy+\int_{Q\cap R_{d_2}}|f(y)-g_{\ep}(y)-f_Q+A_{d_2}|dy
\\
&\quad\leq
\int_{Q}|f(y)-f_Q|dy+\int_{Q\cap R_{d_2}}|f(y)-g_{\ep}(y)-f_Q+A_{d_2}|dy
\\
&\quad\leq
|Q|^{1+\frac{\al}{n}}\ep+\int_{Q\cap R_{d_2}}|f(y)-g_{\ep}(y)-f_Q+A_{d_2}|dy.
\end{align*}

Take a cube $Q_{x_0}\subset (R_{{d_2+1}}\bs \mathring{R_{d_2}})\cap Q$.
By Proposition \ref{proposition, regularity}, \eqref{e-osci dya cub} and \eqref{e-mean value regu}, we obtain
\be
\begin{split}
  |f_Q-A_{d_2}|
  \leq &
  |f_Q-f_{Q_{x_0}}|+|f_{Q_{x_0}}-A_{d_2}|
  \ls|Q|^{\frac{\al}{n}}\ep+C|Q_{x_0}|^{\frac{\al}{n}}\ep
  \ls|Q|^{\frac{\al}{n}}\ep.
\end{split}
\ee
This implies that
\be
\begin{split}
  \int_{Q\cap R_{d_2}}|f_Q-A_{d_2}|dy
  \ls &
  |Q|^{1+\frac{\al}{n}}\ep.
\end{split}
\ee
On the other hand, let $\mathcal D_Q$ be as in Case 1. A similar argument as in Case 1 yields that
\begin{align*}
  \int_{Q\cap R_{d_2}}|f(y)-g_{\ep}(y)|dy
  \leq &
  \sum_{Q_x\in \mathcal D_Q}\int_{Q_x}|f(y)-g_{\ep}(y)|dy
  \ls \sum_{Q_x\in \mathcal D_Q}|Q_x|^{1+\frac{\al}{n}}\ep
  \ls |Q|^{1+\frac{\al}{n}}\ep.
\end{align*}
Combining the above two estimates yields that
\be
\begin{split}
&\int_{Q}|f(y)-g_{\ep}(y)-f_Q+A_{d_2}|dy
\\
&\quad\leq
|Q|^{1+\frac{\al}{n}}\ep+\int_{Q\cap R_{d_2}}|f(y)-g_{\ep}(y)-f_Q+A_{d_2}|dy
\ls|Q|^{1+\frac{\al}{n}}\ep.
\end{split}
\ee
This implies that
\be
\tcoa(f-g_{\ep};Q)
\leq
\frac{1}{|Q|^{1+\frac{\al}{n}}}\int_{Q}|f(y)-g_{\ep}(y)-f_Q+A_{d_2}|dy
\ls\ep.
\ee
Thus, \eqref{e-osci upp bdd} holds.

Now we show \eqref{e-h-e}.
Take a positive $C_c^{\infty}(\mathbb{R}^n)$ function $\va$ supported on $B(0,1)$, satisfying $\|\va\|_{L^1(\rn)}=1$.
Set $\va_{t}(x):=\frac{1}{t^n}\va(\frac{x}{t})$, $t\in (0,\infty)$.
Recall that $g_{\ep}\in Lip_1(\rn)$. Take sufficiently small $r$ such that
\be
2\|g_{\ep}\|_{Lip_1(\rn)}r^{1-\al}<\ep.
\ee
Thus, for $|z|<r$,
\begin{align}\label{proof, 5}
&|(g_{\ep}\ast \va_t)(x+z)-g_{\ep}(x+z)-((g_{\ep}\ast \va_t)(x)-g_{\ep}(x))|\nonumber
\\
&\quad=
\left|\int_{B(0,1)}(g_{\ep}(x+z-ty)-g_{\ep}(x-ty))\va(y)dy\right|+|g_{\ep}(x+z)-g_{\ep}(x)|\nonumber
\\
 &\quad\leq
\|g_{\ep}\|_{Lip_1(\rn)}\cdot \left(\int_{B(0,1)}|z|\va(y)dy+|z|\right)\nonumber
\\
 &\quad=
2\|g_{\ep}\|_{Lip_1(\rn)}|z|=2\|g_{\ep}\|_{Lip_1(\rn)}|z|^{1-\al}|z|^{\al}<\ep|z|^{\al}.
\end{align}
On the other hand, $g_{\ep}$ is uniformly continuous, so
\be
(g_{\ep}\ast \va_t)(x)\longrightarrow g_{\ep}(x)\ \ \ \text{uniformly for all}\ x\in \rn.
\ee
Thus, one can choose sufficiently small $t$ such that
\be
\begin{split}
|(g_{\ep}\ast \va_t)(x+z)-g_{\ep}(x+z)-((g_{\ep}\ast \va_t)(x)-g_{\ep}(x))|<\ep r^{\al}
\end{split}
\ee
uniformly for all $x, z\in \rn$.
From this, for $|z|\geq r$,
\be
\frac{|(g_{\ep}\ast \va_t)(x+z)-g_{\ep}(x+z)-((g_{\ep}\ast \va_t)(x)-g_{\ep}(x))|}{|z|^{\al}}<\ep.
\ee
Combining this and (\ref{proof, 5}), we actually get
\be
\frac{|(g_{\ep}\ast \va_t)(x+z)-g_{\ep}(x+z)-((g_{\ep}\ast \va_t)(x)-g_{\ep}(x))|}{|z|^{\al}}<\ep.
\ee
for all $|z|\neq 0$.
This implies that
\be
\|g_{\ep}\ast \va_t-g_{\ep}\|_{\lipa}<\ep.
\ee
Note that $h_{\ep}:=g_{\ep}\ast \va_t-A_{d_2}\in C_c^{\infty}(\bbR^n)$, and
$\|h_{\ep}-g_{\ep}\|_{\lipa}<\ep$.
We have now completed the proof of \eqref{e-h-e}, which together with \eqref{proof, 6} implies $\cmoa\subset \tcmoa$ for $\alpha\in (0,1)$.

Next, we deal with the case $\al:=1$.
By condition (1) of Definition \ref{def, cmoa}, for any fixed $\ep>0$, there is a constant $a>0$ such that for any $Q$ with $l_Q<2a$,
\be
\co_1(f;Q)<\ep.
\ee

For a fixed point $x$ and a variable point $y$ with $|x-y|<a$, one can choose a suitable cube $Q$ with side length $l_Q\leq 2|x-y|$, such that
\be
Q_{x,y,b}:=x+\frac{bl_Q}{10}Q_0\subset Q,\ \ \ Q_{y,x,b}:=y+\frac{bl_Q}{10}Q_0\subset Q,\ \  b\in (0,1].
\ee
It follows from Proposition \ref{proposition, regularity}  and $\alpha=1$ that
\be
|f_{Q_{x,y,b}}-f_{Q}|\lesssim |Q|^{\frac{1}{n}}\ep\lesssim |x-y|\ep,\ \ \
|f_{Q_{y,x,b}}-f_{Q}|\lesssim |Q|^{\frac{1}{n}}\ep\lesssim |x-y|\ep.
\ee
Letting $b\rightarrow 0$, we obtain
\be
|f(x)-f_{Q}|\lesssim |x-y|\ep,\ \ \ |f(y)-f_{Q}|\lesssim  |x-y|\ep.
\ee
This implies that
\be
|f(x)-f(y)|\lesssim |x-y|\ep.
\ee
Hence,
\be
\ limsup_{y\rightarrow x}\frac{|f(x)-f(y)|}{|x-y|}\lesssim \ep.
\ee
By the arbitrariness of $\ep$, we actually get
\be
\frac{\partial f(x)}{\partial x_j}\equiv 0\ \  \text{for all}\ x\in \rn, j=1,2,\cdots,n.
\ee
This completes the proof for $\al=1$.

\section{Necessity of compactness of commutators}\label{s5}

In this section, in order to deal with the necessity of compact commutators,
the lower bound in Proposition \ref{proposition, lower estimates} is not enough.
So, based on Proposition \ref{proposition, lower estimates},
we establish a further lower bounded estimate providing that $b\in \lipa$.
Next, for $b\in\lipa$, $\Om\in L^{\infty}(\mathbb S^{n-1})$ and any cube $Q$,
we also obtain an upper bound of the weighted $L^q$
norm of $(T_{\Omega,\,\al})_b^m(f)$ over the annulus $2^{d+1}Q\bs 2^dQ$.
Using Theorem \ref{theorem, characterization of cmoa}, the upper and further lower bounds,
and a reduction of $\Om$,
we further present the proof of $(1)\Longrightarrow (2)$ part in Theorem \ref{theorem, characterization of compactness} via a contradiction argument in Subsection \ref{s3.4}.

\subsection{Further lower estimates}\label{s3.2}

In this subsection, we further establish the lower bound fitting for compactness.
Here, although the local mean oscillation is lost,
the continuity of $b\in \lipa$ provides enough information for the distribution of the values of $b$.
This observation makes it possible for us to get the lower bound
for the $L^q$ norm over certain measurable set associated with $Q$.

\begin{proposition}\label{proposition, lower estimates compactness}
Let $\eta_0>0$, $1<p,q<\infty$, $0< \al\leq 1$, $0\leq \b<n$, $m\al+\b<n$, $1/q=1/p-(m\al+\b)/n$ and $m\in \mathbb{Z}^+$.
  Let $w\in \apq$, $b\in \lipa$ be a real-valued function and $\Om$ be a measurable function on $\bbS^{n-1}$, which does not change sign and is not equivalent to zero
  on some open subset of  $\bbS^{n-1}$.
 For any given cube $Q$ with $\tcoa(b;Q)\geq \eta_0$,
 let $\widetilde{P}:=2P$ be the set associated with $\widetilde{Q}:=2Q$ and $\g:=\frac{1}{2^{n+1}}\left(\min\left\{\left(\frac{\eta_0}{4\|b\|_{\lipa}}\right)^{1/\al} \frac{1}{\sqrt{n}},\frac{1}{2}\right\}\right)^n$
 as mentioned in Proposition \ref{proposition, key lower estimates}.
 There are cubes $E\subset 2Q$ and $F\subset 2P$ with $|E|=|F|\geq \widetilde{C} \min\left\{\left(\coa(b;Q)\right)^{n/\al},1\right\}|Q|$,
 where the constant $\widetilde{C}$ is independent of $Q$.
For $f:=(\int_{F}\om(x)^pdx)^{-1/p}\chi_F$
and any measurable set $B$ with $|B|\leq \frac{|E|}{2}$, we have
  \be
  \|(T_{\Omega,\,\b})_b^m(f)\|_{L^q(E\bs B,\om^q)}\geq C\min\left\{\left(\coa(b;Q)\right)^{2n/\al},1\right\}\coa(b;Q)^m,
  \ee
  where the constant $C$ is independent of $Q$.
\end{proposition}
\begin{proof}
Assume $\tcoa(b;Q)\geq \eta_0$.
  By the continuity of $b$, there exist $x_0\in Q$ and $y_0\in P$, such that
\be
|b(x_0)-b(y_0)|=\frac{1}{|Q|}\int_Q|b(x)-b_P|dx\geq |Q|^{\frac{\al}{n}}\tcoa(b;Q).
\ee
Set
\be
L_Q:=\min\left\{\left(\frac{\tcoa(b;Q)}{4\|b\|_{\lipa}}\right)^{1/\al} \frac{l_Q}{\sqrt{n}},\frac{l_Q}{2}\right\},
\ee
and
\be
E:=x_0+L_QQ_0,\ \ \ F:=y_0+L_QQ_0.
\ee
We have
\be
E\subset 2Q,\ \ \ F\subset 2P.
\ee
For any $x\in E$, $y\in F$,
\be
\begin{split}
|b(x)-b(y)|&\geq |b(x_0)-b(y_0)|-|b(x)-b(x_0)|-|b(y)-b(y_0)|
  \\
  &\geq
  |Q|^{\frac{\al}{n}}\tcoa(b;Q)-|x-x_0|^{\al}\|b\|_{\lipa}-|y-y_0|^{\al}\|b\|_{\lipa}
  \\
  &\geq
  |Q|^{\frac{\al}{n}}\tcoa(b;Q)-\frac{\tcoa(b;Q)}{4\|b\|_{\lipa}}l_Q^{\al}\|b\|_{\lipa}-\frac{\tcoa(b;Q)}{4\|b\|_{\lipa}}l_Q^{\al}\|b\|_{\lipa}
  \\
  &\geq
  \frac{|Q|^{\frac{\al}{n}}\tcoa(b;Q)}{2}.
\end{split}
\ee
Again, by the continuity of $b$, $b(x)-b(y)$ does not change sign in $E\times F$.

On the other hand, by the above estimate of $b$ and the fact for $x\in 2Q$,
\be
\begin{split}
|N_x\cap 2P|\leq \g|2P|
\leq \frac{1}{2}\left(\min\left\{\left(\frac{\eta_0}{4\|b\|_{\lipa}}\right)^{1/\al} \frac{1}{\sqrt{n}},\frac{1}{2}\right\}\right)^n|P|
\leq \frac{1}{2}L_Q^n=\frac{1}{2}|F|,
\end{split}
\ee
we obtain that
  \begin{align*}
  &\int_{E\bs B}|(T_{\Omega,\,\b})_b^m(f)(x)|dx\\
  &\quad=
  \left(\int_{F}\om(x)^pdx\right)^{-1/p}\cdot \int_{E\bs B}\left|\int_{F}(b(x)-b(y))^m\frac{\Om(x-y)}{|x-y|^{n-\b}}dy\right|dx
  \\
 &\quad =
  \left(\int_{F}\om(x)^pdx\right)^{-1/p}\cdot \int_{E\bs B}\int_{F}|b(x)-b(y)|^m\frac{|\Om(x-y)|}{|x-y|^{n-\b}}dydx
  \\
  &\quad\geq
  \left(\int_{F}\om(x)^pdx\right)^{-1/p}\cdot \int_{E\bs B}\int_{F\bs (N_x\cap 2P)}|b(x)-b(y)|^m\frac{|\Om(x-y)|}{|x-y|^{n-\b}}dydx
  \\
  &\quad\gtrsim
  \left(\int_{F}\om(x)^pdx\right)^{-1/p}\cdot L_Q^{2n}\cdot |Q|^{\frac{m\al+\b}{n}-1}\tcoa(b;Q)^m
  \\
  &\quad\gtrsim
  \left(\int_{F}\om(x)^pdx\right)^{-1/p}\cdot |Q|^{\frac{m\al+\b}{n}+1}\tcoa(b;Q)^m
  \min\left\{\left(\tcoa(b;Q)\right)^{2n/\al},1\right\}.
  \end{align*}

  Using H\"{o}lder's inequality, we obtain
  \ben
  \begin{split}
  \int_{E\bs B}|(T_{\Omega,\,\b})_b^m(f)(x)|dx
  \leq &
  \left(\int_{E\bs B}|(T_{\Omega,\,\b})_b^m(f)(x)|^q\om^q(x)dx\right)^{1/q}\left(\int_{Q}\om^{-q'}(x)dx\right)^{1/q'}.
  \end{split}
  \een

The above two estimates yield that
\be
\begin{split}
  &\left(\int_{E\bs B}|(T_{\Omega,\,\alpha})_b^m(f)(x)|^q\om^q(x)dx\right)^{1/q}
  \\
  &\quad\geq
  \left(\int_{Q}\om^{-q'}(x)dx\right)^{-1/q'}\cdot
  \left(\int_{F}\om(x)^pdx\right)^{-1/p}\\
  &\quad\quad\times |Q|^{\frac{m\al+\b}{n}+1}\tcoa(b;Q)^m
  \min\left\{\left(\tcoa(b;Q)\right)^{2n/\al},1\right\}
  \\
  &\quad\gtrsim
  \min\left\{\left(\tcoa(b;Q)\right)^{2n/\al},1\right\}\tcoa(b;Q)^m.
\end{split}
\ee
\end{proof}

\subsection{Upper estimates}\label{s3.3}

This part follows by the approach of \cite{GuoWuYang17Arxiv}
with some technique modifications fitting for $\lipa$.

\begin{proposition}\label{proposition, upper estimates}
Let $1<p,q<\infty$, $0< \al\leq n$, $0\leq \b<n$, $m\al+\b<n$, $1/q=1/p-(m\al+\b)/n$, $m\in \mathbb{Z}^+$.
Suppose that $b\in \lipa$, $\Om\in L^{\infty}(\bbS^{n-1})$, $\om\in \apq$.
For any cube $Q$ with $\tcoa(b,Q)\geq \eta_0>0$,
denote by $P, F$ the sets associated with $Q$ mentioned in Proposition \ref{proposition, lower estimates compactness}.
Let $f:=(\int_{F}\om(x)^pdx)^{-1/p}\chi_{F}$.
Then, there exists a positive constant $\d$ such that
\be
\|(T_{\Omega,\,\b})_b^m(f)\|_{L^q(2^{d+1}Q\bs 2^dQ, \om^q)}
\lesssim
2^{-\d dn/p}.
\ee
for sufficient large positive constant $d$, where the implicit constant is independent of $d$, $f$ and $Q$.
\end{proposition}

\begin{proof}
Since $\om^p\in A_p$ and $F\subset 2P$ with $|F|\gtrsim |P|$, we have $\om^p(F)\gtrsim \om^p(2P)\geq \om^p(P)$. Then
\be
f(x)=\left(\int_{F}\om(x)^pdx\right)^{-1/p}\chi_F(x)\lesssim \left(\int_{P}\om(x)^pdx\right)^{-1/p}\chi_{2P}(x).
\ee

A direct calculation yields that
\begin{align}\label{proof, 7}
 & |(T_{\Omega,\,\alpha})_b^m(f)(x)|\nonumber\\
 &\quad  \lesssim
  \left(\int_{P}\om(y)^pdy\right)^{-1/p}\int_{2P}|b(x)-b(y)|^m\frac{|\Omega(x-y)|}{|x-y|^{n-\b}}dy\nonumber
  \\
 &\quad =
  \left(\int_{P}\om(y)^pdy\right)^{-1/p}\int_{2P}|b(x)-b_{2P}+b_{2P}-b(y)|^m\frac{|\Omega(x-y)|}{|x-y|^{n-\b}}dy\nonumber
  \\
  &\quad\leq
  \left(\int_{P}\om(y)^pdy\right)^{-1/p}\sum_{i+j=m}C_m^i|b(x)-b_{2P}|^i\int_{2P}|b_{2P}-b(y)|^j\frac{|\Omega(x-y)|}{|x-y|^{n-\b}}dy.
\end{align}
For sufficiently large $d$,
observe that $|x-y|\sim 2^dl_Q$ for $x\in 2^{d+1}Q\bs 2^dQ$ and $y\in 2P$. By Lemma \ref{l-lip norm equiv}, we deduce that
\begin{align}\label{proof, 8}
  \int_{2P}|b_{2P}-b(y)|^j\frac{|\Omega(x-y)|}{|x-y|^{n-\beta}}dy
\lesssim &
\frac{\|\Om\|_{L^{\infty}(\bbS^{n-1})}}{2^{d(n-\b)}|P|^{1-n/\b}}\int_{2P}|b_{2P}-b(y)|^jdy\nonumber
\\
= &
\frac{\|\Om\|_{L^{\infty}(\bbS^{n-1})}}{2^{d(n-\b)}|P|^{-n/\b}}\frac{1}{|P|}\int_{2P}|b_{2P}-b(y)|^jdy\nonumber
\\
\lesssim &
\frac{\|\Om\|_{L^{\infty}(\bbS^{n-1})}|P|^{\frac{j\al}{n}}}{2^{d(n-\b)}|P|^{-n/\b}}\|b\|_{\lipa}^j.
\end{align}
Since $\om^q\in A^q$, there exists a small positive constant  $\ep\leq \ep_n/[\om^q]_{A_{\infty}}$,
such that
\be
\left(\frac{1}{|\tQ|}\int_{\tQ}\om^{q(1+\ep)}(x)dx\right)^{\frac{1}{1+\ep}}\leq \frac{2}{|\tQ|}\int_{\tQ}\om^q(x)dx\ \ \text{for all cubes}\  \tQ.
\ee
From this and the H\"{o}lder inequality, we obtain
\begin{align}\label{proof, 9}
  &\left\||b-b_{2P}|^i\right\|_{L^q(2^{d+1}Q\bs 2^dQ, \om^q)}\nonumber\\
  &\quad\leq
  \left\||b-b_{2P}|^i\right\|_{L^q(2^{d+1}Q, \om^q)}\nonumber
  \\
  &\quad\leq
  \left(\int_{2^{d+v}P}|b(x)-b_{2P}|^{iq}\om^q(x)dx\right)^{1/q}
  \nonumber\\
  &\quad\leq
  |2^{d+v}P|^{1/q}\left(\frac{1}{|2^{d+v}P|}\int_{2^{d+v}P}|b(x)-b_{2P}|^{iq(1+\ep)'}dx\right)^{\frac{1}{q(1+\ep)'}}\nonumber\\
  &\quad\quad\times
  \left(\frac{1}{|2^{d+v}P|}\int_{2^{d+v}P}\om(x)^{q(1+\ep)}dx\right)^{\frac{1}{q(1+\ep)}}
 \nonumber \\
 &\quad \lesssim
  |2^{d+v}P|^{1/q}\left(\frac{1}{|2^{d+v}P|}\int_{2^{d+v}P}|b(x)-b_{2P}|^{iq(1+\ep)'}dx\right)^{\frac{1}{q(1+\ep)'}}\nonumber\\
  &\quad\quad\times
  \left(\frac{1}{|2^{d+v}P|}\int_{2^{d+v}P}\om(x)^{q}dx\right)^{\frac{1}{q}},
  \end{align}
where $v$ is a positive constant independent of $Q$ such that $2Q\subset 2^vP$.
By the fact
\be
|b(x)-b_{2P}|\lesssim (2^dl_P)^{\al}=2^{d\al}|P|^{\frac{\al}{n}},\ \  x\in 2^{d+v}P,
\ee
we have
\be
\begin{split}
  \left(\frac{1}{|2^{d+v}P|}\int_{2^{d+v}P}|b(x)-b_{2P}|^{iq(1+\ep)'}dx\right)^{\frac{1}{q(1+\ep)'}}
 \lesssim
  2^{id\al}|P|^{\frac{i\al}{n}}.
\end{split}
\ee
Combining this with (\ref{proof, 7}), (\ref{proof, 8}), (\ref{proof, 9}) yields that
\be
\begin{split}
  &\|(T_{\Omega,\,\b})_b^m(f)\|_{L^q(2^{d+1}Q\bs 2^dQ, \om^q)}\\
  &\quad\lesssim
  \sum_{i+j=m}\frac{|2^{d+v}P|^{1/q}|P|^{\frac{j\al}{n}}2^{id\al}|P|^{\frac{i\al}{n}}}{2^{d(n-\b)}|P|^{-n/\b}}\left(\frac{1}{|2^{d+v}P|}\int_{2^{d+v}P}\om(x)^{q}dx\right)^{\frac{1}{q}}
  \left(\int_{P}\om(x)^pdx\right)^{-1/p}
  \\
 &\quad \lesssim
  2^{dn(1/q-1+(\b+m\al)/n)}\left(\frac{1}{|2^{d}P|}\int_{2^{d}P}\om(x)^{q}dx\right)^{\frac{1}{q}}\left(\frac{1}{|P|}\int_{P}\om(x)^pdx\right)^{-1/p}.
\end{split}
\ee
By Lemma \ref{l-Ap weight prop} (v) and (ii), there exists a small constant $\d>0$, such that $\om^p\in A_{p-\delta}$ and
\be
\int_{2^{d}P}\om(x)^pdx\leq 2^{dn(p-\d)}[\om^p]_{A_{p-\d}}\int_{P}\om(x)^{p}dx,
\ee
which implies
\be
\left(\frac{1}{|P|}\int_{P}\om(x)^pdx\right)^{-1/p}
\lesssim
2^{-dn/p}2^{dn(1-\d/p)}\left(\frac{1}{|2^{d}P|}\int_{2^{d}P}\om(x)^{p}dx\right)^{-1/p}.
\ee
Thus,
\begin{align}\label{proof, 10}
  &\|(T_{\Omega,\,\alpha})_b^m(f)\|_{L^q(2^{d+1}Q\bs 2^dQ, \om^q)}
  \nonumber\\
 & \quad\lesssim
  2^{dn(1/q-1+(\b+m\al)/n)}2^{-dn/p}2^{dn(1-\d/p)}\nonumber\\
  &\quad\quad\times
  \left(\frac{1}{|2^{d}P|}\int_{2^{d}P}\om(x)^{q}dx\right)^{\frac{1}{q}}\left(\frac{1}{|2^{d}P|}\int_{2^{d}P}\om(x)^{p}dx\right)^{-1/p}
 \nonumber \\
 & \quad\lesssim
  2^{-\d dn/p}\left(\frac{1}{|2^{d}P|}\int_{2^{d}P}\om(x)^{q}dx\right)^{\frac{1}{q}}\left(\frac{1}{|2^{d}P|}\int_{2^{d}P}\om(x)^{p}dx\right)^{-1/p}
\end{align}
By the definition of $A_{p,q}$, we obtain
\be
\left(\frac{1}{|2^{d}P|}\int_{2^{d}P}\om(x)^qdx\right)^{1/q}\left(\frac{1}{|2^{d}P|}\int_{2^{d}P}\om(x)^{-p'}dx\right)^{1/p'}\lesssim 1.
\ee
This together with the following inequality
\be
1\lesssim \left(\frac{1}{|2^{d}P|}\int_{2^{d}P}\om(x)^{-p'}dx\right)^{1/p'}\left(\frac{1}{|2^{d}P|}\int_{2^{d}P}\om(x)^{p}dx\right)^{1/p}
\ee
yields that
\be
\left(\frac{1}{|2^{d}P|}\int_{2^{d}P}\om(x)^{q}dx\right)^{\frac{1}{q}}\left(\frac{1}{|2^{d}P|}\int_{2^{d}P}\om(x)^{p}dx\right)^{-1/p}\lesssim 1.
\ee
From this and (\ref{proof, 10}), we get the desired estimate
\be
\|(T_{\Omega,\,\b})_b^m(f)\|_{L^q(2^{d+1}Q\bs 2^dQ, \om^q)}
\lesssim
2^{-\d dn/p}.
\ee
\end{proof}

\subsection{Proof of $(1)\Longrightarrow (2)$ in Theorem \ref{theorem, characterization of compactness}}
\label{s3.4}

This part follows by the approach of \cite{GuoWuYang17Arxiv}.
Noting that $\Om\in L^{\infty}(\mathbb S^{n-1})$ is not assumed in Theorem \ref{theorem, characterization of compactness},
a reduction of $\Om$ is needed.

First, we need the following proposition for reduction.
\begin{proposition}\label{proposition, boundedness of multilinear commutator}
Let $1<p,q<\infty$, $0< \al\leq 1$, $0\leq \b<n$, $m\al+\b<n$, $1/q=1/p-(m\al+\b)/n$ and $m\in \mathbb{Z}^+$.
Suppose $r'\in [1,p)$, $\Om\in L^r(\bbS^{n-1})$, $\om^{r'}\in A_{\frac{p}{r'},\frac{q}{r'}}$.
For a  vector-valued function $\vec{b}:=(b_1,b_2,\cdots,b_m)$, $b_j\in \bmoa$, we have
  \be
  \|(T_{\Omega,\,\b})_{\vec b}^mf\|_{L^q(\om^q)}\lesssim \|\Om\|_{L^r(\bbS^{n-1})}\prod_{j=1}^m\|b_j\|_{\bmoa}\|f\|_{L^p(\om^p)},
  \ee
  where for suitable function $f$,
  $$(T_{\Omega,\,\b})_{\vec b}^mf(x):=\int_{\mathbb R^n}\prod_{j=1}^n[b_j(x)-b_j(y)]\frac{\Omega(x-y)}{|x-y|^{n-\beta}}f(y)dy.$$
\end{proposition}
\begin{proof} By Lemma \ref{l-lip norm equiv},
  \be
  \begin{split}
    |(T_{\Omega,\,\b})_{\vec b}^m(f)(x)|
   & \ls
    \prod_{j=1}^m\|b_j\|_{\bmoa}\int_{\rn}\frac{|\Om(x-y)|}{|x-y|^{n-\b-m\al}}|f(y)|dy\\
    &=\prod_{j=1}^m\|b_j\|_{\bmoa}
    \cdot T_{|\Omega|,\,\b+m\al}(|f|)(x).
  \end{split}
  \ee
  Then the desired conclusion follows by a classical result of $T_{|\Omega|,\,\b+m\al}$
  (see \cite[Theorem 3.4.2]{LDY2007}).
\end{proof}

To prove $(1)\Longrightarrow (2)$ in Theorem \ref{theorem, characterization of compactness},
we only need to deal with the case that $b$ is real-valued.
If $(T_{\Omega,\,\beta})_b^{m}$ is a compact operator from $L^p(\om^p)$ to $L^q(\om^q)$, then from Theorem \ref{theorem, necessity of boundedness}, $b\in \bmoa$.
To show $b\in\cmoa$, we use a contradiction
argument. Observe that if $b\notin \cmoa$,  $b$ does not satisfy at least one
of (1)-(3) in Definition \ref{def, cmoa}. We further consider the following three cases.

First suppose that $b$ does not satisfy condition (1) in Definition \ref{def, cmoa}.
There exist $\th_0\in (0,1)$ and a sequence of cubes $\{Q_j\}_{j=1}^{\infty}$ with $|Q_j|\searrow 0$ as $j\rightarrow \infty$, such that
\be
\coa(b;Q_j)\geq \th_0.
\ee
Given a cube $Q$ with $\coa(b;Q)\geq \th_0$, let $E$, $F$ with $|E|=|F|\geq \widetilde{C} \min\left\{\left(\coa(b;Q)\right)^{2n/\al},1\right\}|Q|$
be the cubes mentioned in Proposition \ref{proposition, lower estimates compactness} with $\eta_0=\th_0$.
Let $f:=(\int_{F}\om(x)^pdx)^{-1/p}\chi_F$.
Since $\Om\in L^r(\bbS^{n-1})$, for any $\ep>0$ there exists a function $\Om_{\ep}$ on $\bbS^{n-1}$ such that
\be
  \Om_{\ep}\in L^{\infty}(\bbS^{n-1})\ \ \mbox{and} \  \|\Om-\Om_{\ep}\|_{L^r(\bbS^{n-1})}<\ep.
\ee

Applying Propositions \ref{proposition, lower estimates compactness}, there exists a  positive constant $C_0$
independent of $Q$,
such that
\ben\label{proof, 11}
\|(T_{\Omega,\,\b})_b^m(f)\|_{L^q(E\bs B,\om^q)}\geq 2C_0\min\left\{\left(\coa(b;Q)\right)^{2n/\al},1\right\}\coa(b;Q)^m\ \text{for}\  |B|\leq \frac{|E|}{2}.
\een
Next, by $\Omega-\Omega_\epsilon\in L^r(\mathbb S^{n-1})$, Proposition \ref{proposition, boundedness of multilinear commutator} and $b\in \bmoa$,
we can choose sufficiently small constant $\ep_0>0$ such that
\be
\begin{split}
  \|(T_{\Omega,\,\b})_b^m(f)-(T_{\Omega_{\ep_0},\,\b})_b^m(f)\|_{L^q(\bbR^n, \om^q)}
  \leq &
  C\|\Om-\Om_{\ep_0}\|_{L^r(\bbS^{n-1})}\|b\|_{\bmoa}^m\|f\|_{L^p(\om^p)}
  \\
  \leq &
  C\|\Om-\Om_{\ep_0}\|_{L^r(\bbS^{n-1})}\|b\|_{\bmoa}^m\leq \frac{C_0\th_0^{m+\frac{2n}{\al}}}{2}.
\end{split}
\ee
Then, applying Propositions \ref{proposition, upper estimates} with $\eta_0=\th_0$ for $(T_{\Omega_{\ep_0},\,\b})_b^m$,
there exists a positive constant $d_0$
independent of $Q$,
such that
\be
\|(T_{\Omega_{\ep_0},\,\b})_b^m(f)\|_{L^q(\bbR^n\bs 2^{d_0}Q, \om^q)}
\leq
\frac{C_0\th_0^{m+\frac{2n}{\al}}}{2}=\frac{C_0\min\left\{\th_0^{2n/\al},1\right\}\th_0^m}{2}.
\ee
The above two estimates yield that
\begin{align}\label{proof, 12}
&\|(T_{\Omega,\,\b})_b^m(f)\|_{L^q(\bbR^n\bs 2^{d_0}Q, \om^q)}\nonumber\\
&\quad\leq
\|(T_{\Omega_{\ep_0},\,\b})_b^m(f)\|_{L^q(\bbR^n\bs 2^{d_0}Q, \om^q)}
+\|(T_{\Omega,\,\b})_b^m(f)-(T_{\Omega_{\ep_0},\,\b})_b^m(f)\|_{L^q(\bbR^n, \om^q)}\nonumber
\\
 &\quad\leq
C_0\th_0^{m+\frac{2n}{\al}}.
\end{align}

Take a subsequence of $\{Q_j\}_{j=1}^{\infty}$, also denoted by $\{Q_{j}\}_{j=1}^{\infty}$, such that
\be
\frac{|Q_{j+1}|}{|Q_{j}|}\leq \min\{\widetilde{C}^2\th_0^{\frac{2n}{\al}}/4, 2^{-2d_0n}\}.
\ee
Denote
$
B_j: =\left(\frac{|Q_{j-1}|}{|Q_{j}|}\right)^{\frac{1}{2n}}Q_{j},\ \ j\geq 2.
$
It is easy to check
\be
\left(\frac{|Q_{j-1}|}{|Q_{j}|}\right)^{\frac{1}{2n}}\geq 2^{d_0},\ \
|B_{j+1}|=\left(\frac{|Q_{j+1}|}{|Q_{j}|}\right)^{\frac{1}{2}}Q_{j}\leq \widetilde{C}\th_0^{\frac{n}{\al}}|Q_j|/2\leq |E_j|/2,
\ee
where $E_j$ is the corresponding cube associated with $Q_j$ as mentioned in Proposition \ref{proposition, lower estimates compactness}.
Moreover, for any $k>j$, we have
\be
2^{d_0}Q_k\subset B_k,\ \ |B_k|\leq |E_{j}|/2.
\ee

Denote by $F_j$ the set associated with $Q_j$ as mentioned in Proposition \ref{proposition, lower estimates compactness}.
Let
\be
f_j:=\left(\int_{F_j}\om(x)^pdx\right)^{-1/p}\chi_{F_j}.\
\ee
Again, from (\ref{proof, 11}) and (\ref{proof, 12}),
for any $k>j\geq 1$, we obtain
\be
\|(T_{\Omega,\,\b})_b^m(f_j)\|_{L^q(E_j\bs B_k,\om^q)}\geq
2C_0\min\left\{\left(\coa(b;Q_j)\right)^{2n/\al},1\right\}\coa(b;Q_j)^m
\geq
2C_0\th_0^{m+\frac{2n}{\al}}
\ee
and
\be
\|(T_{\Omega,\,\b})_b^m(f_{k})\|_{L^q(E_j\bs B_k, \om^q)}
\leq
\|(T_{\Omega,\,\b})_b^m(f_{k})\|_{L^q(\bbR^n\bs 2^{d_0}Q_{k}, \om^q)}\leq C_0\th_0^{m+\frac{2n}{\al}}.
\ee
Hence,
\be
\begin{split}
 & \|(T_{\Omega,\,\b})_b^m(f_j)-(T_{\Omega,\,\b})_b^m(f_{k})\|_{L^q(\bbR^n, \om^q)}\\
& \quad \geq
  \|(T_{\Omega,\,\b})_b^m(f_j)-(T_{\Omega,\,\b})_b^m(f_{k})\|_{L^q(E_j\bs B_k, \om^q)}
  \\
&\quad  \geq
  \|(T_{\Omega,\,\b})_b^m(f_j)\|_{L^q(E_j\bs B_k, \om^q)}
  -\|(T_{\Omega,\,\b})_b^m(f_{k})\|_{L^q(E_j\bs B_k, \om^q)}\geq C_0\th_0^{m+\frac{2n}{\al}},
\end{split}
\ee
which leads to a contradiction to the compactness of $(T_{\Omega,\,\b})_b^m$.

A similar contradiction argument is valid for the proof of condition (2), we omit the details here.
It remains to prove $b$ satisfies condition (3) of Definition \ref{def, cmoa}.

Assume that $b$ satisfies (2) but does not satisfy (3).
Hence, there exist $\th_1\in (0,1)$ and a sequence of cube $\{\tQ_j\}_{j=1}^{\infty}$ with $|\tQ_j|\lesssim 1$  such that
\be
\tQ_j\cap R_j=\emptyset,\ \  \mathcal O_\alpha(b,\tQ_j)\geq \th_1,
\ee
where $R_j:=[2^{-j},2^j]^n$.
Denote by $\tE_j$, $\tF_j$ the sets associated with $\tQ_j$
as mentioned in Proposition \ref{proposition, key lower estimates} with $\eta_0=\th_1$.
Let
\be
\widetilde{f}_j:=\left(\int_{\tF_j}\om(x)^pdx\right)^{-1/p}\chi_{\tF_j}.\
\ee

Then, there exists a positive constant $C_1$
independent of $Q$,
such that
\ben\label{proof, 13}
\|(T_{\Omega,\,\b})_b^m(\widetilde{f}_j)\|_{L^q(\widetilde{E}_j,\om^q)}\geq 2C_1\min\left\{\left(\coa(b;\tQ_j)\right)^{2n/\al},1\right\}\coa(b;\tQ_j)^m.
\een

By the similar method as above, there exists a positive constant $d_1$
independent of $\tQ_j$,
such that
\ben\label{proof, 14}
\|(T_{\Omega,\,\b})_b^m(\widetilde{f}_j)\|_{L^q(\bbR^n\bs 2^{d_1}\tQ_j, \om^q)}
\leq
C_1\min\left\{\th_1^{2n/\al},1\right\}\th_1^m=C_1\th_1^{m+\frac{2n}{\al}}.
\een
Take $d_2\geq d_1$ such that $\tE_j\subset 2^{d_2}\tQ_j$,
and a subsequence of $\{\tQ_j\}_{j=1}^{\infty}$, still denoted by $\{\tQ_j\}_{j=1}^{\infty}$, such that
\be
2^{d_2}\tQ_i\cap 2^{d_2}\tQ_j=\emptyset,\ \ i\neq j.
\ee
For any $k\neq j$, note that $2^{d_1}\tQ_k\cap \tE_j\subset 2^{d_2}\tQ_k\cap 2^{d_2}\tQ_j=\emptyset$,
then (\ref{proof, 13}) implise
\be
\begin{split}
\|(T_{\Omega,\,\b})_b^m(\widetilde{f}_j)\|_{L^q(\tE_j\bs 2^{d_1}\tQ_k,\om^q)}
= &\|(T_{\Omega,\,\b})_b^m(\widetilde{f}_j)\|_{L^q(\tE_j,\om^q)}
\\
\geq &
2C_1\min\left\{\left(\coa(b;\widetilde Q_j)\right)^{2n/\al},1\right\}\coa(b;\widetilde Q_j)^m
\geq
2C_1\th_1^{m+\frac{2n}{\al}}.
\end{split}
\ee
From this and (\ref{proof, 14}) we get
\be
\begin{split}
 & \|(T_{\Omega,\,\b})_b^m(\widetilde{f}_j)-(T_{\Omega,\,\b})_b^m(\widetilde{f}_{k})\|_{L^q(\bbR^n, \om^q)}\\
 & \quad\geq
  \|(T_{\Omega,\,\b})_b^m(\widetilde{f}_j)-(T_{\Omega,\,\b})_b^m(\widetilde{f}_{k})\|_{L^q(\tE_j\bs 2^{d_1}\tQ_k, \om^q)}
  \\
 &\quad \geq
  \|(T_{\Omega,\,\b})_b^m(\widetilde{f}_j)\|_{L^q(\tE_j\bs 2^{d_1}\tQ_k, \om^q)}
  -\|(T_{\Omega,\,\b})_b^m(\widetilde{f}_{k})\|_{L^q(\tE_j\bs 2^{d_1}\tQ_k, \om^q)}\geq C_1\th_1^{m+\frac{2n}{\al}},
\end{split}
\ee
which leads to a contradiction to the compactness of $(T_{\Omega,\,\b})_b^m$.

\section{sufficiency of the compactness of commutators}\label{s4}

In this section, we establish the compactness of the iterated commutators mentioned in Theorem \ref{theorem, characterization of compactness}.
We follow some approach in \cite{GuoWuYang17Arxiv}, see also \cite{KrantzLi01JMAAb}.
Here, since the kernel $\Om$ is rough, we will use Proposition \ref{proposition, boundedness of multilinear commutator}
to give a reduction for $\Om$, regaining the smoothness of kernel in the further proof.

Firstly, we recall the following weighted
Fr\'echet-Kolmogorov theorem obtained in \cite{ClopCruz13AASFM}.
\begin{lemma}\label{lemma, criterion of precompact}
  Let $p\in (1,\infty)$ and $\om\in A_p$. A subset $E$ of $L^p(\om)$ is precompact (or totally bounded)
  if the following statements hold:
  \bn[(a)]
  \item $E$ is bounded, i.e., $sup_{f\in E}\|f\|_{L^p(\om)}\lesssim 1$;
  \item $E$ uniformly vanishes at infinity, that is,
  \be
  \lim_{N\rightarrow \infty}\int_{|x|>N}|f(x)|^p\om(x)dx\rightarrow 0,\ \text{uniformly for all}\ f\in E.
  \ee
  \item $E$ is uniformly equicontinuous, that is,
  \be
  \lim_{\r\rightarrow 0}\sup_{y\in B(0,\r)}\int_{\bbR^n}|f(x+y)-f(x)|^p\om(x)dx\rightarrow 0,\ \text{uniformly for all}\ f\in E.
  \ee
  \en
\end{lemma}

\begin{proof}[Proof of $(2)\Longrightarrow (1)$ in Theorem \ref{theorem, characterization of compactness}]
For $\al=1$, we have $(T_{\Omega,\,\b})_b^m=0$ by Theorem \ref{theorem, characterization of cmoa}.

For $\al\in (0,1)$, by the definition of compact operator, we will verify the set
$$A(\Om,b):=\{(T_{\Omega,\,\b})_b^m(f): \|f\|_{L^p(\om^p)}\leq 1\}$$
 is precompact.

Suppose $b\in \cmoa$.
If $r<\infty$, for any $\ep>0$, there exist $b_{\ep}\in C_c^{\infty}(\bbR^n)$ and $\Om_{\ep}\in Lip_1(\bbS^{n-1})$
such that
\be
\|b-b_{\ep}\|_{\bmoa}<\ep\ \ \ and\ \ \ \|\Om-\Om_{\ep}\|_{L^r(\bbS^{n-1})}<\ep.
\ee
From this and Proposition \ref{proposition, boundedness of multilinear commutator}, we obtain
\be
\|(T_{\Omega,\,\b})_b^m-(T_{\Omega,\,\b})_{b_{\ep}}^m\|_{L^p(\om^p)\rightarrow L^q(\om^q)}
\lesssim \ep \|\Om\|_{L^r(\bbS^{n-1})}\|b\|_{\bmoa}^{m-1}
\ee
and
\be
\|(T_{\Omega,\,\b})_{b_{\ep}}^m-(T_{\Omega_{\ep},\,\b})_{b_{\ep}}^m\|_{L^p(\om^p)\rightarrow L^q(\om^q)}
\lesssim \|\Om-\Om_{\ep}\|_{L^r(\bbS^{n-1})}\|b_{\ep}\|_{\bmoa}^{m}.
\ee

If $r=\infty$, then $r'=1$ and hence $\om\in A_{p,q}$.
One can choose a constant $\tilde{r}<\infty$ such that $\tilde{r}'\in (1,p)$ and
$
\om^{\tilde{r}'}\in A_{\frac{p}{\tilde{r}'},\frac{q}{\tilde{r}'}}.
$
For any $\tilde{\ep}>0$ there exist $b_{\tilde{\ep}}\in C_c^{\infty}(\bbR^n)$ and $\Om_{\tilde{\ep}}\in Lip_1(\bbS^{n-1})$
such that
\be
\|b-b_{\tilde{\ep}}\|_{\bmoa}<\tilde{\ep}\ \ \ and\ \ \ \|\Om-\Om_{\tilde{\ep}}\|_{L^{\tilde{r}}(\bbS^{n-1})}<\tilde{\ep}.
\ee
Then, Proposition \ref{proposition, boundedness of multilinear commutator} yields that
\be
\|(T_{\Omega,\,\b})_b^m-(T_{\Omega,\,\b})_{b_{\tilde{\ep}}}^m\|_{L^p(\om^p)\rightarrow L^q(\om^q)}
\lesssim \tilde{\ep} \|\Om\|_{L^{\tilde{r}}(\bbS^{n-1})}\|b\|_{\bmoa}^{m-1}
\ee
and
\be
\|(T_{\Omega,\,\b})_{b_{\tilde{\ep}}}^m-(T_{\Omega_{\tilde{\ep}},\,\b})_{b_{\tilde{\ep}}}^m\|_{L^p(\om^p)\rightarrow L^q(\om^q)}
\lesssim \|\Om-\Om_{\tilde{\ep}}\|_{L^{\tilde{r}}(\bbS^{n-1})}\|b_{\tilde{\ep}}\|_{\bmoa}^{m}.
\ee
Thus, in order to verify the set $A(\Om,b)$ is precompact, or equivalently, totally bounded on $L^q(\om^q)$,
we only need to consider the case of $b\in C_c^\fz(\rn)$ and $\Om\in Lip_1(\bbS^{n-1})$.

Next, take $\va\in C_c^{\infty}(\bbR^n)$ supported on $B(0,1)$ such that $\va=1$ on $B(0,1/2)$, $0\leq \va\leq 1$.
Let $\va_{\d}(x):=\va(\frac{x}{\d})$, $K_{\b}^{\d}(x): =\frac{\Om(x)}{|x|^{n-\b}}\cdot (1-\va_{\d}(x))$,
\be
A(K^{\d},b):=\{(T_{K_{\b}^{\d}})_b^m(f): \|f\|_{L^p(\om^p)}\leq 1\}, \ \
T_{K_{\b}^{\d}}:=\int_{\bbR^n}K_{\b}^{\d}(x-y)f(y)dy\ \ \ \text{for all}\ x\notin \text{supp}f.
\ee

Since $b\in C_c^{\infty}(\bbR^n)$ and $\Om\in L^{\infty}(\mathbb S^{n-1})$, we have
\be
\begin{split}
  |(T_{K_\b^{\d}})_b^mf(x)-(T_{\Omega,\,\b})_b^mf(x)|
  \leq &
  \left|\int_{\bbR^n}(b(x)-b(y))^m\va_{\d}(x-y)\frac{\Om(x-y)}{|x-y|^{n-\b}}f(y)dy\right|
  \\
  \lesssim &
  \int_{|x-y|\leq \d}\frac{|f(y)|}{|x-y|^{n-\b-m}}dy.
\end{split}
\ee
By the usual dyadic decomposition method, we get
\be
\begin{split}
  \int_{|x-y|\leq \d}\frac{|f(y)|}{|x-y|^{n-\b-m}}dy
= &
\sum_{j=0}^{\infty}\int_{2^{-(j+1)}\d \leq|x-y|\leq 2^{-j}\d}\frac{|f(y)|}{|x-y|^{n-\b-m}}dy
\\
\leq &
\sum_{j=0}^{\infty}(2^{-j}\d)^{m(1-\al)}\int_{2^{-(j+1)}\d \leq|x-y|\leq 2^{-j}\d}\frac{|f(y)|}{|x-y|^{n-\b-m\al}}dy
\\
\lesssim &
\sum_{j=0}^{\infty}2^{-jm(1-\al)}\d^{m(1-\al)} M_{\b+m\al}(f)(x)\lesssim \d^{m(1-\al)}M_{\b+m\al}(f)(x),
\end{split}
\ee
where for $\gamma\in(0, n)$, $M_\gamma f(x)$ is the fractional maximal function defined by
$$M_\gamma f(x):=\sup_{Q\ni x}\frac1{|Q|^{1-\frac\gamma n}}\int_Q|f(y)|\,dy.$$
From the above two estimates we obtain
\be
\|(T_{K_\b^{\d}})_b^mf-(T_{\Omega,\,\b})_b^mf\|_{L^q(\om^q)}\lesssim \d^{m(1-\al)}\|M_{\b+m\al}f\|_{L^q(\om^q)}\lesssim \d^{m(1-\al)}\|f\|_{L^p(\om^p)}.
\ee
Since $\d$ can be chosen arbitrarily small, we only need to verify $A(K^{\d},b)$ is totally bounded,
where $\d>0$, $b\in C_c^{\infty}(\bbR^n)$.
Now, we finish the reduction argument for $b$ and $\Om$.

Since $\Om\in Lip_1(\bbS^{n-1})$, one can see that
\be
|K_\b^{\d}(x)-K_\b^{\d}(x')|\lesssim \frac{|x-x'|}{|x|^{n-\b+1}},\ \ \ 2|x-x'|\leq |x|.
\ee
We only need to check the conditions (a)-(c) of Lemma \ref{lemma, criterion of precompact}
for $A(K^{\d},b)$.

Without loss of generality, we assume that $b$ is supported in a cube $Q$ centered at the origin.
By the boundedness of $(T_{K_{\b}^{\d}})_b^m$, $A(K^{\d},b)$ is a bounded set in $L^q(\om^q)$, which verifies condition (a).

For $x\in (2Q)^c$,
\be
\begin{split}
  |(T_{K_{\b}^{\d}})_b^m(f)(x)|
  = &
  \left|\int_{\bbR^n}(b(y))^mK_{\b}^{\d}(x-y)f(y)dy\right|
  \\
  \lesssim &
  \frac{\|b\|_{L^{\infty}}^m}{|x|^{n-\b}}\int_{Q}|f(y)|dy
  \\
  \leq &
  \frac{\|b\|_{L^{\infty}}^m}{|x|^{n-\b}}\|f\|_{L^p(\om^p)}\left(\int_{Q}\om^{-p'}(x)dx\right)^{1/p'}.
\end{split}
\ee
Take $N>2$,
\begin{align}\label{proof, 15}
  &\left(\int_{(2^NQ)^c}|(T_{K_{\b}^{\d}})_b^m(f)(x)|^q\om(x)^qdx\right)^{1/q}\nonumber\\
  &\quad\lesssim
  \left(\int_{(2^NQ)^c}\frac{\om(x)^q}{|x|^{q(n-\b)}}dx\right)^{1/q}\left(\int_{Q}\om^{-p'}(x)dx\right)^{1/p'}.
\end{align}
Since $\om^q\in A_{\frac{q(n-\b-m\al)}{n}}$, we obtain
\be
\int_{2^{d}Q}\om(x)^qdx\leq 2^{dq(n-\b-m\al)}[\om^q]_{A_{\frac{q(n-\b-m\al)}{n}}}\int_{Q}\om(x)^{q}dx,
\ee
which implies
\be
\int_{2^{d+1}Q\bs 2^dQ}\frac{\om(x)^q}{|x|^{q(n-\b)}}dx
\lesssim
\frac{2^{dq(n-\b-m\al)}}{2^{dq(n-\b)}}=2^{-dqm\al}.
\ee
This and (\ref{proof, 15}) yield that
\be
\begin{split}
  \left(\int_{(2^NQ)^c}|(T_{K_{\b}^{\d}})_b^m(f)(x)|^q\om(x)^qdx\right)^{1/q}
  \lesssim &
  \left(\sum_{j=0}^{\infty}\int_{2^{N+j+1}Q\bs 2^{N+j}Q}\frac{\om(x)^q}{|x|^{q(n-\b)}}dx\right)^{1/q}
  \\
  \lesssim &
  \left(\sum_{j=0}^{\infty}2^{-(N+j)qm\al}\right)^{1/q}
  =
  2^{-Nm\al}\left(\sum_{j=0}^{\infty}2^{-jqm\al}\right)^{1/q},
\end{split}
\ee
which tends to zero as $N$ tends to infinity. This proves condition (b).

It remains to prove that $A(K^{\d},b)$ is equicontinuous in $L^q(\om^q)$. Assume that $\|f\|_{\lpwp}=1$ and take $z\in \bbR^n$ with $|z|\leq \frac{\d}{8}$. Then
\be
\begin{split}
  &(T_{K_\b^{\d}})_b^m(f)(x+z)-(T_{K_\b^{\d}})_b^m(f)(x)
  \\
  &\quad=
  \int_{\bbR^n}(b(x+z)-b(y))^m(K_{\b}^{\d}(x+z-y)-K_{\b}^{\d}(x-y))f(y)dy
  \\
  &\quad\quad +
  \int_{\bbR^n}\big((b(x+z)-b(y))^m-(b(x)-b(y))^m\big)K_{\b}^{\d}(x-y)f(y)dy
  \\
&\quad=: I_1(x,z)+I_2(x,z).
\end{split}
\ee
We start the estimate of the first term.
Observing that $K_{\b}^{\d}(x+z-y)$ and $K_{\b}^{\d}(x-y)$ both vanish when $|x-y|\leq \frac{\d}{4}$,
then
\begin{align*}
  |I_1(x,z)|
  \leq &
  \int_{|x-y|\geq \d/4}|b(x+z)-b(y)|^m|K_{\b}^{\d}(x+z-y)-K_{\b}^{\d}(x-y)|\cdot|f(y)|dy
  \\
  \lesssim &
  \int_{|x-y|\geq \d/4}\frac{|z|}{|x-y|^{n-\b-m\al+1}}|f(y)|dy
  \\
  \lesssim &
  \sum_{j=0}^{\infty}\int_{2^{j-2}\d\leq |x-y|\leq 2^{j-1}\d}\frac{|z|}{|x-y|^{n-\b-m\al+1}}|f(y)|dy
  \\
  \leq &
  \sum_{j=0}^{\infty}\frac{2^{2-j}|z|}{\d}
  \int_{2^{j-2}\d\leq |x-y|\leq 2^{j-1}\d}\frac{1}{|x-y|^{n-\b-m\al}}|f(y)|dy
  \\
  \lesssim &
  \sum_{j=0}^{\infty}\frac{2^{2-j}|z|}{\d}M_{\b+m\al}(f)(x)
  \lesssim \frac{|z|}{\d}M_{\b+m\al}(f)(x).
\end{align*}

Hence,
\be
\|I_1(\cdot,z)\|_{L^q(\om^q)}\lesssim \frac{|z|}{\d}\|M_{\b+m\al}f\|_{L^q(\om^q)}
\lesssim \frac{|z|}{\d}\|f\|_{L^p(\om^p)}\leq \frac{|z|}{\d}.
\ee
Finally,
\be
\begin{split}
  |I_2(x,z)|
  = &
 \left| \int_{\bbR^n}\big((b(x+z)-b(y))^m-(b(x)-b(y))^m\big)K_{\b}^{\d}(x-y)f(y)dy\right|
  \\
  \lesssim &
  |z|\int_{\bbR^n}|K_{\b}^{\d}(x-y)|\cdot|f(y)|dy\lesssim \int_{\bbR^n}|z|\frac{|f(y)|}{|x-y|^{n-\b-m\al}}dy.
\end{split}
\ee
Hence,
\be
\|I_{2}(\cdot,z)\|_{L^q(\om^q)}\lesssim |z|\|I_{\b+m\al}f\|_{L^q(\om^q)}\lesssim |z|\|f\|_{L^p(\om^p)}\leq |z|.
\ee
It follows from above estimates of $I_1$, $I_{2}$ that
\be
  \|(T_{K_\b^{\d}})_b^m(f)(\cdot+z)-(T_{K_\b^{\d}})_b^m(f)(\cdot)\|_{L^q(\om^q)}\rightarrow 0,
\ee
as $|z|\rightarrow 0$, uniformly for all $f$ with $\|f\|_{L^p(\om^p)}\leq 1$.
\end{proof}

\end{document}